\documentclass[11pt]{article}
\usepackage[totalwidth=13.0cm,totalheight=20.0cm]{geometry}
\usepackage{latexsym,amsthm,amsmath,amssymb,url}
\usepackage[ruled, linesnumbered]{algorithm2e}

\usepackage{tikz,authblk}
\usetikzlibrary{decorations.pathreplacing}

\newcommand{\2}{\vspace{0.15cm}}
\newcommand{\sX}[1]{\hspace{-0.1cm}{\footnotesize #1} \hspace{-0.2cm}}
\newcommand{\tX}[1]{\hspace{-0.05cm}{\footnotesize #1} \hspace{-0.1cm}}
\newcommand{\smallQED}{\hfill {\tiny ($\Box$)}}
\newcommand{\Mac}{\mathrm{mac}}
\renewcommand{\circ}{\mathrm{circ}}

\newtheorem{theorem}{Theorem}[section]
\newtheorem{lemma}[theorem]{Lemma}
\newtheorem{corollary}[theorem]{Corollary}
\newtheorem{definition}[theorem]{Definition}
\newtheorem{proposition}[theorem]{Proposition}

\newtheorem{openProblem}[theorem]{Open Problem}

\newcommand{\AZ}[1]{{#1}}

\begin{document}
	
	\title{Bounds on Maximum Weight Directed Cut}
\author{
Jiangdong Ai\thanks{School of Mathematical Sciences and LPMC, Nankai University. {\tt JiangdongAi95@gmail.com}.} \hspace{2mm} Stefanie Gerke\thanks{Department of Mathematics. Royal Holloway University of London.  {\tt stefanie.gerke@rhul.ac.uk}.} \hspace{2mm} Gregory Gutin\thanks{Department of Computer Science. Royal Holloway University of London. {\tt g.gutin@rhul.ac.uk}.} \\ Anders Yeo\thanks {Department of Mathematics and Computer Science, University of Southern Denmark. {\tt andersyeo@gmail.com}, and Department of Mathematics, University of Johannesburg.} \hspace{2mm} Yacong Zhou\thanks{Department of Computer Science. Royal Holloway University of London. {\tt Yacong.Zhou.2021@live.rhul.ac.uk}.} }
	\date{}
	\maketitle
	
	\begin{abstract}
		We obtain lower and upper bounds for the maximum weight of a directed cut in the classes of
		weighted digraphs and weighted acyclic  digraphs as well as in some of their subclasses. We compare our 
		results with those obtained for the maximum size of a directed cut in unweighted digraphs. In particular, we 
		show that a lower bound obtained by Alon, Bollob\'as, Gy\'{a}f\'{a}s, Lehel and Scott (J Graph Th 55(1) (2007)) for unweighted acyclic digraphs can be extended to weighted digraphs
		with the maximum length of a cycle being \AZ{bounded by a constant and the weight of every arc being at least one}. 
		We state a number of open problems.
	\end{abstract}
	
	\pagestyle{plain}
	\section{Introduction}\label{sec:intro}   
	
	Let $D=(V(D),A(D),w_D)$ be a weighted digraph with weight function $w_D:\ A(D) \to \mathbb{R}_+,$ where $\mathbb{R}_+$ is the set of non-negative reals. For a subgraph $H$ of $D$, $w_D(H)$ denotes the sum of weights of arcs in $H.$ (In what follows, we will omit subscripts identifying directed or undirected graphs if these graphs are clear from the context.) Let $X$ and $Y$ be a partition of $V(D).$ Then the {\em directed cut} (or {\em dicut} for short) $(X,Y)$ of $D$ is the bipartite subgraph of $D$ induced by the arcs going from $X$ to $Y$; its weight is denoted by $w(X,Y).$ 
	The aim of the {\sc  Maximum Weight Directed Cut} problem is to find a directed cut of $D$ of maximum weight. 
	This weight will be denoted by ${\rm mac}(D).$ Analogously, the weight of a maximum cut in an undirected graph $G$ will be denoted by ${\rm mac}(G).$ 
	
	Clearly, {\sc  Maximum Weight Directed Cut} is a generalization of {\sc  Maximum Weight Cut} for undirected graphs and thus {\sf NP}-hard \cite{Karp1972}. While there is a large number of papers on lower bounds for  {\sc  Maximum Cut} (see e.g. \cite{Alon,ABKS,Edw,Lau,PT1986}), where the weight of each edge is 1, 
	as far as we know there are only two papers on lower bounds for {\sc  Maximum Weight Cut}: the well-known paper \cite{PT1986} of Poljak and Turz\'ik and the very recent paper \cite{GY2021} of Gutin and Yeo. While there are papers on {\sc  Maximum Directed Cut}, see e.g. Alon, Bollob\'{a}s, Gy\'{a}f\'{a}s, Lehel and  Scott \cite{Alon2006}; Lehel, Maffray and Preissman \cite{LMP2009}; Xu and Yu \cite{XuYu}; and Chen, Gu and Li \cite{CGL2014}, as far as we know, our paper is the first on lower bounds for {\sc  Maximum Weight Directed Cut}. \smallskip

	%Recently, Gutin and Yeo  studied lower bounds of ${\rm mac}(G)$ and generalized some classic results of  \cite{PT1986}. However, these two papers are the only literatures we know working on a weighted version of maximum cuts problems. A natural question one may ask is: for any weighted digraph $D$, what is maximum weighted cut of $D$? Although this {\scshape Maximum Weighted Directed Cut} problem sounds more general than the {\scshape Maximum Weighted Cut} problem since every graph can be transformed into a digraph by replacing every edge with two symmetric arcs. We will show that they are mutually polynomially reducible (see Section \ref{sec: reductions}). As finding a maximum weighted directed cut is an NP-hard problem, it is interesting to study the lower bound of ${\rm mac}(D)$.
	
	%As far as we know, no one has worked on the lower bound of maximum weighted directed cuts. But there are some papers working on the lower bound of maximum directed cuts for some classes of digraphs. In 2006, Alon et al. \cite{Alon2006} studied the lower bound of ${\rm mac}(D)$ when $D$ is an acyclic digraph or digraph with degree restrictions. For other related results, we refer the readers to \cite{CGL2014, LMP2009}. In this paper, we study the lower bound of the maximum weighted directed cuts for both general digraphs and acyclic digraphs. Some of our results generalized those of Alon et al. \cite{Alon2006}.
	
	% Before overviewing the main results  of this paper, we will introduce additional notation and terminology.  
	
	%In what follows, we will often shorten ``directed cut" to "dicut".     
	
	For any $v\in V(D)$, let $w^+(v)$ be the sum of weights of arcs leaving $v$, $w^-(v)$ the sum of weights of arcs entering $v$, and $r(v)=w^+(v)-w^-(v)$. Note that $\sum_{v\in V(D)}r(v)=0.$ Let \[ r^+(D)=\sum_{r(x)>0} r(x)=\sum_{x \in V(D)} |r(x)|/{2}.\]
	Note that for any cut $(X,Y)$, we have 
	\begin{equation} \label{eq:plus}
		\sum_{x\in X} r(x) =w_D(X,Y)-w_D(Y,X) \leq \mathrm{mac}(D).
	\end{equation}
	By choosing $X$ to be the set of all vertices $v$ with $r(v)>0$, one immediately obtains that 
	\begin{equation} \label{eq:easybound} r^+(D) \leq \mathrm{mac}(D). \end{equation} %Let $w(D)$ be the sum of the weights of all arcs in $D$.

In the rest of this section we first provide an overview of the paper and then additional terminology and notation. 	
	 
\paragraph{Paper overview}	
	In Section \ref{sec: general bounds}, \AZ{we prove basic extensions to \eqref{eq:easybound} and use these bounds to show some relations between $\mathrm{mac}(D)$ and the weight of maximum cuts of the underlying graph of $D$. Then, we use random partitions to prove lower bounds for some classes of digraphs. For example, we obtain a tight lower bound on the maximum weight of a dicut for weighted digraphs with bounded maximum semidegrees, which generalizes a result of Alon, Bollob\'as, Gy\'{a}f\'{a}s, Lehel and Scott  \cite{Alon2006}. We also show  a useful analog of a result proved by Gutin and Yeo \cite{GY2021} for dicuts. Later in the paper, we will use the analog in the proof of our main result.}  
	
	Let $\theta(D)=r^+(D)/w(D)$, and  for all $0 \leq \theta \leq 1$, let  $l(\theta)$ be
	\[
	l(\theta)
	= \left\{ \begin{array}{lll} \vspace{0.1cm}
		\left( \frac{1}{4} + \frac{\theta^2}{4(1-2\theta)} \right) & & \mbox{if $\theta < 1/3$;}  \\
		\theta & & \mbox{if $\theta \geq 1/3$.} \\
	\end{array} \right.
	\]
	\AZ{The function $l(\theta)$ is depicted in Fig. \ref{fig:l}.} 
	\AZ{We show that if $\theta(D)\geq 1/3$ then the bound $\Mac(D)\geq l(\theta(D)) w(D)$ (equivalent to  (\ref{eq:easybound})) is best possible as there are digraphs attaining this bound.}
	For $\theta(D)<1/3$ we prove that $\Mac(D)\geq l(\theta(D)) w(D)$ and that for all $\epsilon >0$, there exists a digraph  $D$ with $\Mac(D) <(1+\epsilon) l(\theta(D)) w(D)$.
	
	\begin{figure}
		\begin{center}
			\tikzstyle{vertexB}=[rectangle,draw, minimum size=8pt, scale=0.9, inner sep=0.9pt]
			\tikzstyle{vertexA}=[circle,draw, minimum size=8pt, scale=0.9, inner sep=0.9pt]
			\tikzstyle{vertexAs}=[circle,draw, minimum size=8pt, scale=0.6, inner sep=0.9pt]
			\begin{tikzpicture}[scale=4.5]
				
				%  \draw (0,5) rectangle (3,9); \node at (1.5,7) {$X_A$};
				
				\node at (-0.1,0.5) {{\tiny 0.5}};
				\node at (-0.1,1) {{\tiny 1.0}};
				
				\node at (0.5,-0.08) {{\tiny 0.5}};
				\node at (1.0,-0.08) {{\tiny 1.0}};
				
				\node at (1.2,0) {{\footnotesize $\theta$}};

				\draw[line width=0.03cm] (-0.02,0.1) to (0.02,0.1);
				\draw[line width=0.03cm] (-0.02,0.2) to (0.02,0.2);
				\draw[line width=0.03cm] (-0.02,0.3) to (0.02,0.3);
				\draw[line width=0.03cm] (-0.02,0.4) to (0.02,0.4);
				\draw[line width=0.03cm] (-0.02,0.5) to (0.02,0.5);
				\draw[line width=0.03cm] (-0.02,0.6) to (0.02,0.6);
				\draw[line width=0.03cm] (-0.02,0.7) to (0.02,0.7);
				\draw[line width=0.03cm] (-0.02,0.8) to (0.02,0.8);
				\draw[line width=0.03cm] (-0.02,0.9) to (0.02,0.9);
				\draw[line width=0.03cm] (-0.02,1.0) to (0.02,1.0);

				\draw[line width=0.03cm] (0.1,-0.02) to (0.1,0.02);
				\draw[line width=0.03cm] (0.2,-0.02) to (0.2,0.02);
				\draw[line width=0.03cm] (0.3,-0.02) to (0.3,0.02);
				\draw[line width=0.03cm] (0.4,-0.02) to (0.4,0.02);
				\draw[line width=0.03cm] (0.5,-0.02) to (0.5,0.02);
				\draw[line width=0.03cm] (0.6,-0.02) to (0.6,0.02);
				\draw[line width=0.03cm] (0.7,-0.02) to (0.7,0.02);
				\draw[line width=0.03cm] (0.8,-0.02) to (0.8,0.02);
				\draw[line width=0.03cm] (0.9,-0.02) to (0.9,0.02);
				\draw[line width=0.03cm] (1,-0.02) to (1,0.02);

				\draw[->, line width=0.03cm] (0.0,0.0) to (0.0,1.1);
				\draw[->, line width=0.03cm] (0.0,0.0) to (1.1,0.0);

				\draw[line width=0.02cm] (0.0,0.25) to (0.01,0.25002551020408165); 
				\draw[line width=0.02cm] (0.01,0.25002551020408165) to (0.02,0.2501041666666667); 
				\draw[line width=0.02cm] (0.02,0.2501041666666667) to (0.03,0.2502393617021277); 
				\draw[line width=0.02cm] (0.03,0.2502393617021277) to (0.04,0.25043478260869567); 
				\draw[line width=0.02cm] (0.04,0.25043478260869567) to (0.05,0.25069444444444444); 
				\draw[line width=0.02cm] (0.05,0.25069444444444444) to (0.060000000000000005,0.2510227272727273); 
				\draw[line width=0.02cm] (0.060000000000000005,0.2510227272727273) to (0.07,0.25142441860465115); 
				\draw[line width=0.02cm] (0.07,0.25142441860465115) to (0.08,0.2519047619047619); 
				\draw[line width=0.02cm] (0.08,0.2519047619047619) to (0.09,0.2524695121951219); 
				\draw[line width=0.02cm] (0.09,0.2524695121951219) to (0.09999999999999999,0.253125); 
				\draw[line width=0.02cm] (0.09999999999999999,0.253125) to (0.10999999999999999,0.25387820512820514); 
				\draw[line width=0.02cm] (0.10999999999999999,0.25387820512820514) to (0.11999999999999998,0.25473684210526315); 
				\draw[line width=0.02cm] (0.11999999999999998,0.25473684210526315) to (0.12999999999999998,0.25570945945945944); 
				\draw[line width=0.02cm] (0.12999999999999998,0.25570945945945944) to (0.13999999999999999,0.25680555555555556); 
				\draw[line width=0.02cm] (0.13999999999999999,0.25680555555555556) to (0.15,0.2580357142857143); 
				\draw[line width=0.02cm] (0.15,0.2580357142857143) to (0.16,0.25941176470588234); 
				\draw[line width=0.02cm] (0.16,0.25941176470588234) to (0.17,0.2609469696969697); 
				\draw[line width=0.02cm] (0.17,0.2609469696969697) to (0.18000000000000002,0.26265625); 
				\draw[line width=0.02cm] (0.18000000000000002,0.26265625) to (0.19000000000000003,0.26455645161290325); 
				\draw[line width=0.02cm] (0.19000000000000003,0.26455645161290325) to (0.20000000000000004,0.26666666666666666); 
				\draw[line width=0.02cm] (0.20000000000000004,0.26666666666666666) to (0.21000000000000005,0.2690086206896552); 
				\draw[line width=0.02cm] (0.21000000000000005,0.2690086206896552) to (0.22000000000000006,0.2716071428571429); 
				\draw[line width=0.02cm] (0.22000000000000006,0.2716071428571429) to (0.23000000000000007,0.27449074074074076); 
				\draw[line width=0.02cm] (0.23000000000000007,0.27449074074074076) to (0.24000000000000007,0.27769230769230774); 
				\draw[line width=0.02cm] (0.24000000000000007,0.27769230769230774) to (0.25000000000000006,0.28125); 
				\draw[line width=0.02cm] (0.25000000000000006,0.28125) to (0.26000000000000006,0.28520833333333334); 
				\draw[line width=0.02cm] (0.26000000000000006,0.28520833333333334) to (0.2700000000000001,0.28961956521739135); 
				\draw[line width=0.02cm] (0.2700000000000001,0.28961956521739135) to (0.2800000000000001,0.29454545454545455); 
				\draw[line width=0.02cm] (0.2800000000000001,0.29454545454545455) to (0.2900000000000001,0.30005952380952383); 
				\draw[line width=0.02cm] (0.2900000000000001,0.30005952380952383) to (0.3000000000000001,0.3062500000000001); 
				\draw[line width=0.02cm] (0.3000000000000001,0.3062500000000001) to (0.3100000000000001,0.3132236842105264); 
				\draw[line width=0.02cm] (0.3100000000000001,0.3132236842105264) to (0.3200000000000001,0.32111111111111124); 
				\draw[line width=0.02cm] (0.3200000000000001,0.32111111111111124) to (0.3300000000000001,0.3300735294117648); 
				\draw[line width=0.02cm] (0.3300000000000001,0.3300735294117648) to (0.3333333333333333,0.3333333333333333); 
				\draw[line width=0.02cm] (0.3333333333333333,0.3333333333333333) to (1.0,1.0); 
				%\node at (1,0) {\mbox{Â }};
			\end{tikzpicture}
			\caption{The function $l(\theta)$.} \label{fig:l}
	\end{center} \end{figure}
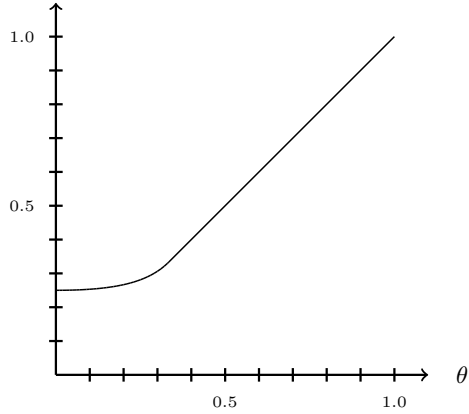

	\smallskip

	In Section \ref{sec: bounds for acyclic directed multigraphs}, we consider acyclic digraphs.	
	Let $h(m)$ be the maximum integer such that each acyclic digraph with $m$ arcs has a dicut of size at least $h(m)$. The main results of Alon, Bollob\'as, Gy\'{a}f\'{a}s, Lehel and Scott  in \cite{Alon2006} are\footnote{Note that  \cite{Alon2006} has a typo in Theorem 2 claiming that $h(m)=m/4+\Omega(m^{2/3})$ instead of $h(m)=m/4+\Omega(m^{3/5})$; see  \cite{Alon2006p}, where Theorem 2 is stated correctly \cite{G}.}
	$h(m)=m/4+O(m^{4/5})$ and $h(m)=m/4+\Omega(m^{3/5})$. Here we consider similar functions for acyclic directed multigraphs and weighted acyclic digraphs. 
	Let $h_{\mu}(m)$ be the maximum integer such that each acyclic directed multigraph with $m$ arcs has a dicut of size at least $h_{\mu}(m)$. \AZ{Let ${\cal D}_{\omega}$ be the set of weighted digraphs in which every arc has weight at least 1 and let ${\cal D}_{\omega}(w)$ be
	the subset of ${\cal D}_{\omega}$ such that $w=w(D)$ for each $D\in {\cal D}_{\omega}(w).$
	}
Then for each $w\ge 1$, let $h_{\omega}(w)$ be the supremum of reals $g$ such that every acyclic \AZ{$D\in {\cal D}_{\omega}(w)$} has a dicut of weight at least $g$. 
	Note that every weighted \AZ{acyclic} digraph $H$ can be transformed into a weighted \AZ{acyclic digraph $D_H\in {\cal D}_{\omega}(w')$} for some $w'\ge w(H)$ such that a maximum weight dicut in $H$ with minimal number of arcs is a maximum dicut in $D_H$ and a maximum weight dicut in $D_H$ is a maximum weight dicut in $H$. Indeed, this holds if $D_H$ is obtained from $H$ by deleting all arcs of weight zero and dividing all arc-weights by the minimum (positive) arc-weight, if the minimum arc-weight is smaller than 1. 
	
	Clearly, $h_{\omega}(m)\le h_{\mu}(m)\le h(m)$ for each integer $m\ge 1$. For any $w\geq 1$, we will show that $h_{\omega}(w)=w/4+O(w^{3/4})$ by proving $h_{\mu}(m)=m/4+O(m^{3/4})$ and using monotonicity of $h_\omega (w)$. We will also prove that $h_{\omega}(w)=w/4+\Omega(w^{3/5})$. \AZ{Note that this lower bound does not hold for weighted digraphs $D$ with the total weight $w=w(D)<1.$ Indeed, if $w<1$ and $0<\alpha<1$, then $w^\alpha>w$.} Note that even though our upper bound for $h_{\mu}(m)$ is better than that of $h(m)$, they are not comparable. In contrast, our lower bound extends that of \cite{Alon2006} from acyclic digraphs to weighted acyclic digraphs. \AZ{Our proof is different from that in \cite{Alon2006} as we could not extend the proof of \cite{Alon2006} to weighted digraphs.}
	The asymptotic value of functions $h$, $h_{\mu}$ and $h_{\omega}$ remains an interesting open problem.

	In Section \ref{sec: acyclic digraphs with bounded path length}, we study the best possible lower bound for weighted acyclic digraphs with bounded path length. Namely, for every integer $\nu\geq 2$, we study the largest possible coefficient $c_\nu>0$ such that for every weighted acyclic digraph $D$ in which the longest path has $\nu$ vertices, we have ${\rm mac}(D)\ge c_\nu w(D)$. The exact value of $c_\nu$ is determined when $\nu\leq 11$ (see Appendix). However the main goal of this section is to give general bounds for $c_\nu$, and we will show that $c_\nu=1/4+O(\nu^{-1/2})$ and $c_\nu=1/4+\Omega(\nu^{-2/3})$. Determining an asymptotic value of $c_{\nu}$ remains an open problem. 
	
	In Section~\ref{Generalization of Theorem}, we obtain the main result of this paper:
the lower bound in Section~\ref{sec: bounds for acyclic directed multigraphs} for \AZ{acyclic} $D\in {\cal D}_{\omega}(w)$, i.e., $h_\omega(w)=w/4+\Omega(w^{3/5})$ still holds asymptotically when the maximum length of a cycle \AZ{is bounded by a constant. In fact, we prove a stronger result. Namely, we show that for any $0<\alpha<1$, if the lower bound $\mathrm{mac}(D)= w(D)/4+\Omega(w(D)^\alpha)$ holds for all weighted acyclic digraphs $D\in {\cal D}_{\omega}(w)$, then the same bound holds asymptotically even if we allow any circle of length at most a constant. In our proof, we use a theorem of Bondy \cite{Bondy} which states that the chromatic number of the underlying graph of a digraph $D$ is at most the length of a longest cycle in $D$.}
	
\paragraph{Additional Terminology and Notation}	
A vertex $v$ of a digraph $D$ is a {\em source} ({\em sink}, respectively) if $d^-(v)=0$ ($d^+(v)=0$, respectively). All paths and cycles in digraphs are directed. {The {\em length} of a cycle or path is the number of its arcs. A $k$-{\em path} ($k$-{\em cycle}, respectively) is a path (a cycle, respectively) of length $k$.}
The {\em underlying graph} of a weighted digraph $D=(V(D),A(D),w_D)$ is a weighted graph $UG(D)$ with the same vertex set as $D$, where two vertices $x$ and $y$ of $G$ are adjacent  if there is an arc between $x$ and $y$ ($UG(D)$  has no multiple arcs)
	and if $xy\in A(D),$ but $yx\not\in A(D)$ then $w_G(xy)=w_D(xy)$ and if $xy\in A(D)$ and $yx\in A(D)$ then $w_G(xy)=w_D(xy)+w_D(yx).$ A digraph $D$ is {\em connected} if $UG(D)$ is connected. A subgraph $H$ of a digraph $D$ is a {\em connected component} of $D$, if $UG(H)$ is a connected component of $UG(D).$ The {\em order} of a directed or undirected graph is the number of its vertices. Terminology and notation on digraphs not defined in this paper can be found in \cite{BJG09,BJG18}.
	
	%In Section \ref{sec: reductions}, we show that the problem of finding maximum weighted dicuts is equivalent to the problem of finding maximum weighted cuts in undirected graphs in the sense of polynomial reductions.
	
	\section{Bounds for Arbitrary Weighted Digraphs}\label{sec: general bounds}
	
	This section is partitioned into two subsections. In the first subsection, we prove a number of basic lower and upper bounds for  ${\rm mac}(D)$ and in the second subsection, we study  a lower bound for  ${\rm mac}(D)$ using parameter $\theta(D)$ introduced in Section \ref{sec:intro}.

	\subsection{Basic Bounds}

	\begin{theorem}\label{thm:basic}
		Let $D=(V(D),A(D),w_D)$ be a weighted digraph and let $G=UG(D).$
		\begin{description}
			\item[(a)] If the minimum weight of a dicut of $D$ is $k$, then $r^+(D)+k \leq {\rm mac}(D)$.
			\item[(b)] $\frac{{\rm mac}(G)}{2} \leq {\rm mac}(D) \leq \frac{{\rm mac}(G) + r^+(D)}{2}$.
			\item[(c)] $\frac{{\rm mac}(G)/2 +r^+(D)}{2} \leq {\rm mac}(D) \leq \frac{{\rm mac}(G) + r^+(D)}{2}$.
			\item[(d)] Let the chromatic number of $G$ be $\chi $. If $\chi$ is even, 	then 
			\[ \left(\frac{1}{4}+\frac{1}{4 (\chi-1)}\right)w(D) \leq  {\rm mac}(D);\]  and if $\chi$ is odd, then 
			\[ \left(\frac{1}{4}+\frac{1}{4 \chi}\right)w(D) \leq  {\rm mac}(D).\]
		\end{description}
	\end{theorem}
	\begin{proof}
		For (a), let $X$ contain all vertices, $x$, with $r(x) \geq 0$. Then by \eqref{eq:plus}
		\[ r^+(D) = w_D(X,Y) - w_D(Y,X) \leq w_D(X,Y) - k \leq {\rm mac}(D) - k.\]
		For (b),  we have by definition that for any cut $(X,Y)$ 
		\[
		w_G(X,Y) = w_D(X,Y) + w_D(Y,X).
		\]
		In particular this equality is true for the maximum cut in $G$ and hence in this case $w_D(X,Y)$ or  $w_D(Y,X)$ has to be at least ${\rm mac}(G)/2$ which shows the lower bound.
		For the upper bound, adding this equality to  \eqref{eq:plus}, we get  
		\[ w_G(X,Y)+\sum_{x \in X} r(x) = 2w_D(X,Y).\]
		As $\sum_{x \in X} r(x) \leq r^+(D)$ and $w_G(X,Y) \leq {\rm mac}(G)$, this
		implies that 
		\[ 2w_D(X,Y) \leq {\rm mac}(G) + r^+(D)\]
		for any cut and in particular when $(X,Y)$ is a cut of $D$ of the maximum weight. 
		
		The lower bound of (c) follows by adding \eqref{eq:easybound} and (b) together and dividing by two, while the upper bound is the same as in (b).
		
		For (d) consider a partition of $V(D)$ into $\chi$ independent sets $V_1,\ldots,V_\chi$. We construct an auxilary weighted undirected complete graph  $G'$ with $\chi$ vertices $v_1,\ldots,v_\chi$ and for $i<j$ the edge $\{v_i,v_j\}$ has weight $w'(v_iv_j)= w(V_i,V_j) + w(V_j,V_i)$. Note that $w'(G')=w(D)$ and that every cut  $(X',Y')$ in $G'$ corresponds to two cuts $(X,Y)$  and $(Y,X)$  in $D$ such that $w'(X',Y')=w(X,Y) +w(Y,X)$. If $\chi =2t$ is even then we partition $V(G')$ randomly and uniformly into two parts $X'$ and $Y'$ of size $t$.
		For an edge $e\in E(G')$, let the $I(e) = w'(e)$ if one endpoint of $e$ is in $X'$ and the other in $Y'$, and $I(e)=0$ otherwise. Then 
		\begin{align*} \mathbb{E}(w'(X',Y')) &=\mathbb{E} \left( \sum_{e\in E(G')} I(e)	  \right) = 	 \sum_{e\in E(G')} \mathbb{E}( I(e))\\
			& = \sum_{e\in E(G')} w'(e) \frac{ 2 {{2t-2}\choose {t-1}} }{ { {2t}\choose t} } = \frac{t}{2t -1}  w'(G') = \left( \frac{1}{2} + \frac{1}{ 2\chi -2} \right) w(D).
		\end{align*}
		If $\chi =2t+1$ is odd then we partition $V(G')$ randomly and uniformly into two parts $X'$ and $Y'$ of size $t+1$ and $t$ respectively.
		For an edge $e\in E(G')$, let  $I(e) = w'(e)$ if one endpoint of $e$ is in $X'$ and the other in $Y'$, and $I(e)=0$ otherwise. Then 
		\[  \mathbb{E}(w'(X',Y'))=  \sum_{e\in E(G')} w'(e) \frac{2  {{2t-1}\choose {t-1}} }{ { {2t+1}\choose t} } =  \frac{t+1}{2t +1}w'(G') = \left(\frac{1}{2}+ \frac{1}{2\chi}   \right)  w(D).\]
		In either case there exists a cut $(X'',Y'')$ in $G'$ with $w'(X'',Y'')\geq \mathbb{E}(w(X',Y'))$, and the result now follows since there exist corresponding cuts $(X,Y)$ and $(Y,X)$ such that $w'(X'',Y'')=w(X,Y) +w(Y,X)$. 	 
	\end{proof}

		Part (b) of Theorem \ref{thm:basic} shows that if $r^+(D)=0$, then ${\rm mac}(D) = \frac{{\rm mac}(G)}{2}$. Let $D$ be any strong digraph (without arc weights) and let $G$ be the underlying graph of $D$. We can assign arc weights to $D$ such that $r^+(D)=0$ as follows. Initially let all weights be zero. For every arc, $uv \in A(D)$ let $C_{uv}$ be a cycle containing $uv$ and add one to the weight of all arcs in $C_{uv}$. After doing this for all arcs $uv$ we note that all weights are positive and $r^+(D)=0$. So for digraph $D$ with these weights, we have ${\rm mac}(D) = \frac{{\rm mac}(G)}{2}$. Note that ${\rm mac}(D) \geq \frac{{\rm mac}(G)}{2}$ (by (b)). Thus, the best bounds for general digraphs or general strong digraphs we can get (if we do not restrict the arc weights in any way) are exactly the same bounds we get for the underlying graph of $D$ divided by 2.
	
	Part (d) of Theorem~\ref{thm:basic} immediately allows us to obtain lower bounds on some well studied graph classes. 
	For example, for any integer $k>0$, a graph $G$ is said to be {\em $k$-degenerate} if the minimum degree of any induced subgraph of it is at most $k$. It is well known that if a graph is $k$-degenerate, then its chromatic number is at most $k+1$. If $D=(V,A,w)$ is a weighted digraph with maximum out-degree $\Delta^+(D) \leq k$ or maximum in-degree $\Delta^-(D) \leq k$, then the underlying graph of $D$ is $2k$-degenerate. Indeed, assume $\Delta^+(D) \leq k$ and let $G$  be the underlying  graph of $D$. Since $\Delta^+(D) \leq k$, for any induced subdigraph $D'$ of $D$, there always exists a vertex $v\in V(D')$ with in-degree at most $k$, which implies $d_{G[V(D')]}(v)=d_{D'}^+(v)+d_{D'}^-(v)\leq 2k$. Hence, $G$ is $2k$-degenerate.
	The same argument works with $\Delta^+$ and $\Delta^-$ exchanged. Thus we have the following proposition. 
	
	\begin{proposition} \label{maxDegD}
		Let $D=(V,A,w)$ be a weighted digraph with $\Delta^+(D) \leq k$ or $\Delta^-(D) \leq k$.
		Then ${\rm mac}(D) \geq (\frac{1}{4}+\frac{1}{8k+4})w(D)$.
	\end{proposition}
		Let $T_k$ be the $k$-regular tournament and $K_{2k+1}$ be the complete graph of order $2k+1$, where the weight of every arc of these two graphs is one. Note that $r^+(T_k)=0$ and therefore ${\rm mac}(T_k) = \frac{{\rm mac}(K_{2k+1})}{2}$. Then we have
	\[
	\mathrm{mac}(T_k) =\frac{\mathrm{mac}(K_{2k+1})}{2}= \frac{k(k+1)}{2} = k(2k+1) \frac{k+1}{4k+2} = \frac{k+1}{4k+2}w(T_k).
	\]
	
	So, Proposition~\ref{maxDegD} is tight.

	%\section{Digraphs with Bounded Path Length}\label{sec: bounded path length}
	
	It is not hard to give a tight lower bound for the maximum dicut of digraphs, in which the order of the maximum path is $\nu$, by using the following Gallai-Hasse-Roy-Vitaver Theorem \cite{Gallai1968,Hasse,Roy1967,Vitaver}. 
	\begin{theorem}\cite{Gallai1968,Hasse,Roy1967,Vitaver}\label{Galli-Roy}
		Every digraph $D$ contains a directed path with $\chi(D)$ vertices.
	\end{theorem}
	\begin{theorem}\label{cut(r)}
		If the number of vertices $\nu$ in the longest path in $D$ is odd \textup{(even, respectively)}, then $D$ has a dicut with weights at least $(\frac{1}{4}+\frac{1}{4\nu})w(D)$ \textup{($(\frac{1}{4}+\frac{1}{4(\nu-1)})w(D)$, respectively)}.
	\end{theorem}
	Again, it is easy to check this bound is tight for regular tournament.
	
		Let $\mathcal{B}(D)$ denote the set of bipartite subdigraphs $R$ of $D$ such that for every connected component $R_1$ of $R$ with bipartition $(X_1,X_2)$, both $X_1$ and $Y_1$ induce independent sets in $D$.
		
		\begin{lemma}\label{lem: induced bipartite digraph}
			Let $D=(V,A,w)$ be a weighted digraph. If $R\in \mathcal{B}(D)$, then $\mathrm{mac}(D)\geq \frac{w(D)}{4}+\frac{w(R)}{4}$.
		\end{lemma} 
		\begin{proof}
			
			Let $R_1$, $\dots$, $R_t$ be the connected components of $R$ and $(X_i,Y_i)$ be the bipartition of $R_i$, $i\in [t]$. We will create a random partition $(X,Y)$ of $V(D)$ as follows. For each $i\in [t]$, we let $X_i$ belong to $X$ and $Y_i$ belong to $Y$ with probability $1/2$ and we let $X_i$ belong to $Y$ and $Y_i$ belong to $X$ with probability $1/2$. Thus, $X_i$ and $Y_i$ will never belong to the same set in the partition $(X,Y)$. For any vertex that is not in $V(R)$ we assign it to $X$ with probability $1/2$ and to $Y$ with probability $1/2$. 
			
			We note that every arc in $A(R)$ belongs to the dicut $(X,Y)$ with probability $1/2$ and every other arc in $D$ belongs to the dicut $(X,Y)$ with probability at least $1/4$. So, the average weight of the partition $(X,Y)$ is at least $\frac{w(R)}{2}+ \frac{w(D)-w(R)}{4}=\frac{w(D)+w(R)}{4}$, which proves the theorem.
		\end{proof}
		
		Since any matching in $D$ is clearly in $\mathcal{B}(D)$, the following result holds.
		
		\begin{corollary} \label{lowerBoundMatching}
			Let $D=(V,A,w)$ be a weighted digraph and let $M$ be a matching in $D$. Then ${\rm mac}(D) \geq \frac{w(D)}{4} + \frac{w(M)}{4}$.
		\end{corollary}

	\subsection{Bounds for Given $\theta(D)$}

	Recall that $\theta(D)=r^+(D)/w(D)$. We will give bounds on ${\rm mac}(D)$ in terms of $\theta(D)$ and $w(D)$. Recall that \[
		l(\theta)
		= \left\{ \begin{array}{lll} \vspace{0.1cm}
			\left( \frac{1}{4} + \frac{\theta^2}{4(1-2\theta)} \right) & & \mbox{if $\theta < 1/3$;}  \\
			\theta & & \mbox{if $\theta \geq 1/3.$} \\
		\end{array} \right.
		\]
	Let us discuss a motivation for studying lower bounds on  ${\rm mac}(D)$ in terms of $\theta(D)$. If $r^+(D)/w(D)=0$ then the problem is equivalent to the {\sc Maximum Weighted  Cut} problem for the underlying graph (see the remark after the proof of Theorem \ref{thm:basic}).
	Thus, the best possible bound that can be obtained is ${\rm mac}(D) \geq w(D)/4$.
	If $r^+(D)/w(D)=1$ then the weighted maximum cut  of $D$ contains all arcs in $D$, so the problem is easy and ${\rm mac}(D) = w(D)$.
	
	So, it seems natural to find best possible bounds in the case when $0 < r^+(D)/w(D) < 1$. Multiplying all weights in $D$ by a constant $c$ increases
	$r^+(D)$, $w(D)$ and ${\rm mac}(D)$ by a factor of $c$. Thus, the interesting parameter is $r^+(D)/w(D)$ as this does not change if we multiply all weights
	by a given constant. So we want to find the best possible bounds for ${\rm mac}(D)$ of the form 
	${\rm mac}(D) \geq f(\theta(D)) \cdot  w(D)$, for some function $f$.
	\begin{lemma} \label{boundLowerf}
		${\rm mac}(D) \geq l(\theta(D)) \cdot w(D)$.
	\end{lemma}
	
	\begin{proof}
		
		Note that for $\theta(D)\geq 1/3$, the bound is simply \eqref{eq:easybound}. So for the remainder we assume  $\theta < 1/3$.
		% Note that $\theta = \left( \frac{1}{4} + \frac{\theta^2}{4(1-2\theta)} \right)$ implies that $\theta=1/3$. 
		Let $D$ be any weighted digraph and let $\theta=\theta(D)$. 		Let $R^+$contain all vertices, $x$, in $D$ with $r(x)>0$ and let $R^-=V(D)\setminus R^+$. 
		That is, for all $y\in R^-$ we have $r(y)\leq 0$.

		Let $\bar{p }= \frac{\theta}{2(1-2\theta)}$ and place any vertex from $R^+$ into $X$ with probability $(1/2 + \bar{p})$ and any vertex in 
		$R^-$ into $X$ with probability $(1/2 - \bar{p})$. Let $Y=V(D) \setminus X$. For an arc $a=uv$ 
		let  $I(a) = w(a)$ if  $u\in X$ and $v\in Y$, and $I(a)=0$ otherwise. 
		Then 
		\begin{align*}
			\mathbb{E}(w(X,Y)) &=\mathbb{E} \left( \sum_{a\in A(D)} I(a)	  \right) = 	 \sum_{a\in A(D)} \mathbb{E}( I(a))\\
			&= \sum_{u,v\in R^+}\mathbb{E}( I(uv))+  \sum_{u,v\in R^-}\mathbb{E}( I(uv))+ \sum_{u\in R^+ \atop v \in R^-}\mathbb{E}( I(uv)) + \sum_{u\in R^- \atop v\in R^+}\mathbb{E}( I(uv)) \\
			&= \sum_{u,v\in R^+}w(uv)\left(\frac{1}{4}-\bar{p}^2\right)+ \sum_{u,v\in R^-} w(uv)\left(\frac{1}{4}-\bar{p}^2\right)+ \\
			&+  \sum_{u\in R^+ \atop v \in R^-} w(uv) \left(\frac12 +\bar{p}\right)^2 + \sum_{u\in R^- \atop v\in R^+} w(uv) \left(\frac12 - \bar{p}\right)^2\\
			&=  \left( \frac{1}{4}  - \bar{p}^2  \right)  w(D)  + (2\bar{p}^2 + \bar{p}) w(R^+, R^-) +   (2\bar{p}^2 -  \bar{p}) w(R^-, R^+) \\
			&=    \left( \frac{1}{4}  - \bar{p}^2  \right)  w(D) + 2 \bar{p}^2 (w(R^+,R^-)+w(R^-,R^+))+ \\ 
			&+ \bar{p}(w(R^+,R^-) - w(R^-,R^+))	.
		\end{align*}
		As $w(R^+,R^-) - w(R^-,R^+)=r^+(D)$ and $r^+(D)=\theta \cdot w(D)$,

		%		We will now compute the average weight of 
		%		the cut $(X,Y)$. Let $uv$ be an arbitrary arc in $A(D)$, and consider the following cases.
		%		
		%		\begin{description}
			%			\item [Case 1, $u,v \in R^+$:] The arc, $uv$, belongs to the cut $(X,Y)$ with probability $(1/2+p)(1/2-p)= 1/4-p^2$.
			%			\item [Case 2, $u,v \in R^-$:] The arc, $uv$, belongs to the cut $(X,Y)$ with probability $(1/2+p)(1/2-p)= 1/4-p^2$.
			%			\item [Case 3, $u \in R^+$ and $v \in R^-$:] The arc, $uv$, belongs to the cut $(X,Y)$ with probability $(1/2+p)(1/2+p)= 1/4+p+p^2$.
			%			\item [Case 4, $u \in R^-$ and $v \in R^+$:] The arc, $uv$, belongs to the cut $(X,Y)$ with probability $(1/2-p)(1/2-p)= 1/4-p+p^2$.
			%		\end{description}
		%		
		%		So the average weight of the cut $(X,Y)$ is the following, as $w(R^+,R^-) - w(R^-,R^+)=r^+(D)$.
		
		%		\[
		%		\begin{array}{rcl}
			%			\mathbb{E}(w(X,Y)) & = & \frac{w(D)}{4} - p^2 w(D) + 2 p^2 (w(R^+,R^-)+w(R^-,R^+)) \\
			%			& & + p(w(R^+,R^-) - w(R^-,R^+)) \\
			%			& \geq & \frac{w(D)}{4} - p^2 w(D) + 2p^2 r^+(D) + pr^+(D) \\
			%		\end{array}
		%		\]
		%		As $r^+(D)=\theta \cdot w(D)$ this implies the following.
		\[
		\begin{array}{rcl}
			\mathbb{E}(w(X,Y)) & \geq & \frac{w(D)}{4} + \bar{p}^2 \left( - w(D) + 2 \theta w(D) + \frac{\theta w(D)}{\bar{p}} \right) \\
			& = & \frac{w(D)}{4} + \left( \frac{\theta}{2(1-2\theta)} \right)^2  w(D)  \left( - 1 + 2 \theta  + \frac{\theta \cdot 2(1-2\theta)}{\theta} \right) \\
			& = & \frac{w(D)}{4} +  \frac{\theta^2}{4(1-2\theta)^2}  w(D)  \left( - 1 + 2 \theta  + 2(1-2\theta) \right) \\
			& = & \frac{w(D)}{4} +  \frac{\theta^2}{4(1-2\theta)}  w(D).  \\
		\end{array}
		\]
		This implies that there exists a cut of weight at least $ \left( \frac{1}{4} + \frac{\theta^2}{4(1-2\theta)} \right)\cdot w(D)$, as desired.
	\end{proof}
	The following lemma will help us to analyse the cut size of certain weighted digraphs.
	
	\begin{lemma} \label{lem:maxg}
		Let $Q$ and $k$ be positive reals and  let $g(x,y)=Qxy + x(k-x)/2 + y(k-y)/2$.  If $0< Q \leq 1/2$ then the function $g(x,y)$  is maximized over $x,y \in [0,k]$ when  $x=y=\frac{k}{2(1-Q)}$ and 
		$g(\frac{k}{2(1-Q)},\frac{k}{2(1-Q)}) = \frac{k^2}{4(1-Q)}$. If $1/2<Q<1$ then the function $g(x,y)$ is maximized  over $x,y \in [0,k]$ when  $x=y=k$  and $g(k,k) = Qk^2$.
		
	\end{lemma}
	\begin{proof}
		We have 
		\[
		%\left\{ 
		\begin{array}{lll} \vspace{0.1cm}
			\frac{\partial g}{\partial x}=Qy+\frac{k}{2}-x,  \\
			\frac{\partial g}{\partial y}=Qx+\frac{k}{2}-y. \\
		\end{array}
		%\right.
		\]\\
		Thus, the critical point is $(\frac{k}{2(1-Q)},\frac{k}{2(1-Q)})$, and $g(\frac{k}{2(1-Q)},\frac{k}{2(1-Q)})=\frac{k^2}{4(1-Q)}$. Since $g(x,y)$ is continuous and its first partial derivatives always exist, the optimal value either lies on the boundary or at a critical point. Therefore, we only need to compare the value at the  critical point to those on the boundary. Since $g(x,y)$ is a symmetric function, the cases below will be considered separately for $x=k$ and $x=0$.  If $x=k$ then observe that for $Q>1/2$ the maximum of $g(k,y)$ is attained at $y=k$ and $g(k,k)=Qk^2$ and for $Q\le 1/2$ the maximum of $g(k,y)$ is attained for $y=k(Q+1/2)$ and $g(k,k(Q+1/2))=k^2\frac{Q^2+Q+1/4}{2}.$ If $x=0$ then the maximum of $g(x,y)$ is attained at $(0,k/2)$ and $g(0,k/2)=k^2/8.$
		
		If $Q \leq 1/2$, then we have $\frac{k}{2(1-Q)}\leq k$ and $\frac{k^2}{4(1-Q)}\geq \max\{\frac{Q^2+Q+1/4}{2}k^2,\frac{k^2}{8}\}$ which proves the first part. 
		If $Q > 1/2$, the $\frac{k}{2(1-Q)}>k$ and therefore the results follows from $Qk^2\geq \frac{k^2}{8}$. 
		\end{proof}
	
	\begin{theorem} \label{boundUpperf}
		Let $\theta$ satisfy $0 \leq \theta \leq 1$. 
		
		\begin{itemize}
			\item If  $0 \leq \theta \leq 1/3$, then for every $\epsilon>0$ there exists a 
			digraph $D$, with $\theta(D)=\theta$, which satisfies 
			
			\[
			{\rm mac}(D) < w(D) \cdot (1+ \epsilon) l(\theta).
			\]
			
			\item If $1/3 < \theta \leq 1$, then there exists a 
			digraph $D$, with $\theta(D)=\theta$, which satisfies ${\rm mac}(D)=\theta \cdot w(D)$ ($=l(\theta)\cdot w(D)$).
		\end{itemize}
	\end{theorem}
	
	\begin{proof}
		If $\theta=1$ then ${\rm mac}(D) = w(D) = r^+(D)$ always holds, so assume that $\theta <1$. 
		Let $D_k$ be a digraph consisting of two vertex disjoint regular tournament, $A_k$ and $B_k$, of order $k$ and containing 
		all arcs from $A_k$ to $B_k$ and no arcs from $B_k$ to $A_k$. Define $Q$ as follows.
		
		\[
		Q = \frac{\theta(1-1/k)}{1-\theta}.
		\]
		
		%We let $k$ be large enough so $Q \leq 1/2$, which is possible as $\theta < 1/3$.
		Let the weight of every arc from $A_k$ to $B_k$ be $Q$ and let the weight of
		each arc in $A_k$ and in $B_k$ be one.
		
		Let $(X,Y)$ be a maximum cut in $D_k$ and let $x = |V(A_k) \cap X|$ and $y=|V(B_k) \cap Y|$. Then the following holds.
		
		\begin{itemize}
			\item ${\rm mac}(D_k) = Qxy + x(k-x)/2 + y(k-y)/2$, as the cut contains $x(k-x)/2$ arcs from $A_k$ and $y(k-y)/2$ arcs from $B_k$ (since $d^+_T(x)=d^-_T(x)$ for every $x\in V(T)$, where $T\in \{A_k,B_k\}$, for every $S\subseteq V(T)$ the number of arcs leaving $S$ is equal to the number of arcs entering $S$).
			\item $r^+(D_k) = Qk^2$, as $r(a)=kQ$ for all $a \in V(A_k)$ and $r(b)=-kQ$ for all $b \in V(B_k)$.
			\item $w(D_k) = 2 \cdot {k \choose 2} + Qk^2 = k^2 - k + Qk^2$.
			\item $\theta(D_k) = r^+(D_k)/w(D_k) = Qk^2 /(k^2 - k + Qk^2) = Q/(1+Q-1/k)$.
			\item $\theta(D_k) = \theta$, as we defined $Q = \theta(1-1/k)/(1-\theta)$ which is equivalent to $\theta(1+Q-1/k)=Q$ and therefore holds by the above statement.
		\end{itemize}
		Assume that $0 \leq \theta \leq 1/3$.  Note that $Q< 1/2$ for any  $k$. Choose $k$ large enough  such that 
		\[  \frac{1}{4(1-Q^2) - 4(1-Q)/k}  \leq (1+\epsilon)  \frac{1}{4(1-Q^2) }.\]
		Let $\theta=\theta(D_k)$. Then by Lemma~\ref{lem:maxg}, our choice of $k$ and the definition of $Q$, we have 
		\[
		\begin{array}{rcl}
			{\rm mac}(D_k) & \leq & \frac{k^2}{4(1-Q)} \\
			& = & (k^2 - k + Qk^2) \cdot \left( \frac{k^2}{4(1-Q)(k^2 - k + Qk^2)} \right) \\
			& = & w(D) \left( \frac{1}{4(1-Q)(1 - 1/k + Q)} \right) \\
			& = & w(D) \left( \frac{1}{4(1-Q^2) - 4(1-Q)/k} \right) \\
			&\leq & (1+\epsilon) w(D) \left( \frac{1}{4} + \frac{\theta^2}{4(1-2\theta)} \right) .
		\end{array}
		\]

		%\[
		%\begin{array}{rcl}
		%	\frac{1}{4(1-Q^2)}  & = & \frac{1}{4 \left(1- \left( \frac{\theta(1-1/k)}{1-\theta} \right)^2 \right)} \\
		%	& \leq & \frac{1}{4 \left(1- \left( \frac{\theta}{1-\theta} \right)^2 \right)} \\
		%	& = & \frac{1}{4}Â + \left(    \frac{(1-\theta)^2}{4\left((1-\theta)^2 - \theta^2 \right)}    - \frac{1}{4}Â  \right) \\
		%	& = & \frac{1}{4}Â + \frac{1}{4} \left(    \frac{1-2\theta+\theta^2}{1-2\theta}    - \frac{1-2\theta}{1-2\theta}Â  \right) \\
		%	& = & \frac{1}{4}Â + \frac{1}{4} \left(    \frac{\theta^2}{1-2\theta} \right) \\
		%	& = & \frac{1}{4}Â +  \frac{\theta^2}{4(1-2\theta)} \\
		%	%& = & \frac{1}{4}Â + \frac{\theta^2}{4} \left( 1 + \frac{1}{1-2\theta} - \frac{1-2\theta}{1-2\theta} \right) \\
		%	%& = & \frac{1}{4}Â + \theta^2 \left( \frac{1}{4} + \frac{\theta}{2(1-2\theta)} \right) \\
		%\end{array}
		%\]

		If $1/3 < \theta \leq 1$, then  we choose $k$ large enough that $Q>1/2$. By Lemma~\ref{lem:maxg}  and our earlier observations
		\[ {\rm mac}(D_k) = Qk^2=  r^+(D_k)  = \theta \cdot w(D_k).\]  	\end{proof}

		\begin{definition} \label{maxDegG}
			For every $\theta\in [0,1]$, let $f(\theta)$  be the supremum of all reals $g$ such that  ${\rm mac}(D) \geq g \cdot  w(D)$ holds for all 
			weighted digraphs $D$ with $\theta=\theta(D)$.
		\end{definition}

The next theorem follows from Lemma \ref{boundLowerf} and Theorem \ref{boundUpperf}.
	
	\begin{theorem} \label{boundf}
		We have	$f(\theta) = l(\theta)$.
	\end{theorem}
	%\marginpar{\tiny Move forward? Why $mac(D) \geq 1/4$}
	
	\section{Bounds for Weighted Acyclic Digraphs}\label{sec: bounds for acyclic directed multigraphs}

	Recall that $h_{\mu}(m)$ is the maximum integer such that each acyclic directed multigraph with $m$ arcs has a dicut of size at least $h_{\mu}(m)$. Let ${\cal D}_{\omega}(w)$ be the set of weighted digraphs $D$ of weight $w=w(D)$ such that each arc of $D$ has weight at least one. Then for each $w\ge 1$, let $h_{\omega}(w)$ be the supremum of reals $g$ such that every acyclic $D\in {\cal D}_{\omega}(w)$ has a dicut of weight at least $g$. Clearly, $h_{\omega}(m)\le h_{\mu}(m)$ for each integer $m\ge 1$.  In this section, we focus on giving bounds for $h_\omega(m)$. We first show in Theorem \ref{thm:upperB} that $h_\mu (m)=\frac{m}{4}+O(m^{3/4})$ implying $h_\omega (w)=\frac{w}{4}+O(w^{3/4})$. Then, in Theorem \ref{thm:dag} we prove that $h_\omega (w)=\frac{w}{4}+\Omega(w^{3/5})$.

	\begin{theorem} \label{thm:upperB}
		%There exists a constant $k_1$, such that for infinitely many $r$ there exists an acyclic digraph $D_r$ with longest path length $r$
		%and ${\rm mac}(D_r) \leq \frac{w(D_r)}{4} + k_1 \cdot  w(D_r)^{3/4}$.
		There exists a constant $k_1$, such that for every integer $m\ge 1$ there exists an acyclic directed multigraph $D'_{m}$ with $m$ arcs
		such that 
		\[
		{\rm mac}(D'_m) \leq   \frac{m}{4} +  k_1 m^{3/4}.
		\]
	\end{theorem}
	
	\begin{proof}
		We will first construct a directed multigraph $D$ on $n\geq 4$ vertices as follows. Let $V(D)=\{v_1,v_2,\ldots,v_n\}$ and
		let $q=\lfloor \sqrt{n} \rfloor$.
		Let $T_i$ denote the transitive tournament on $q$ vertices $\{v_i,v_{i+1},\ldots,v_{i+q-1}\}$ where
		all indices are taken modulo $n$ and all arcs point forward in the ordering $(v_i,v_{i+1},\ldots,v_{i+q-1})$.
		%We will determine $q \geq 2$ and $n$ later.
		Let $A(D) = \cup_{i=1}^n A(T_i)$ and note that $D$ is a regular directed multigraph. Furthermore, $|A(D)|= n {q \choose 2}$.
		
		Let $(X,Y)$ be an optimal cut in $D$. As $(X,Y)$ contains at most $q^2/4$ edges from the underlying graph of $T_i$ for each
		$i=1,2,\ldots,n$ (as the underlying graph of $T_i$ is a complete graph on $q$ vertices), we note that $(X,Y)$ contains at most 
		$nq^2/4$ edges from the underlying multigraph of $D$. As $D$ is regular we note that 
		${\rm mac}(D) \leq nq^2/8$.
		
		Let $D_n$ be obtained from $D$ by deleting all arcs $v_j v_i$ with $j>i$. We note that we delete
		exactly $s(q-s)$ arcs from $T_{n+1-s}$ for each $s=1,2,3, \ldots, q-1$. Therefore, 
		\begin{align*}
			|A(D_n)| & =  |A(D)| - \sum_{s=1}^{q-1} (qs-s^2)\\
			&  =  \frac{nq(q-1)}{2} - q \left( \sum_{s=1}^{q-1} s \right) + \left( \sum_{s=1}^{q-1} s^2 \right) \\ 
			& =  \frac{nq^2}{2} - \frac{nq}{2} - q \frac{q(q-1)}{2} + \frac{(q-1)q(2q-1)}{6} \\ 
			& =  \frac{nq^2}{2} - \frac{nq}{2} - \frac{3q^3-3q^2}{6} + \frac{2q^3 -3q^2 + q}{6} \\ 
			& =  \frac{nq^2}{2} - \frac{nq}{2} - \frac{q^3}{6} + \frac{q}{6} .
		\end{align*}
                Isolating the $\frac{nq^2}{2}$-term and then 
		dividing through by four implies that
		
		\[
		{\rm mac}(D_n)  \leq  {\rm mac}(D) 			
			 \leq  \frac{nq^2}{8} 
			 =  \frac{|A(D_n)|}{4} + \frac{nq}{8} + \frac{q^3}{24} - \frac{q}{24} 
			 \leq  \frac{|A(D_n)|}{4} + \frac{3nq+q^3}{24} . 		\]
		By the definition of $q$, we have $n= \alpha q^2$ for some $\alpha\ge 1$,  and thus
		\[
		{\rm mac}(D_n)  \leq   \frac{|A(D_n)|}{4} + q^3 \times \frac{3\alpha +1}{24} .
		\]
		Also note that 
		\[
		|A(D_n)|^{3/4} = \left( \frac{nq^2}{2} - \frac{nq}{2} - \frac{q^3}{6} + \frac{q}{6} \right)^{3/4}
		= q^3 \left( \frac{\alpha}{2} - \frac{\alpha}{2q} - \frac{1}{6q} + \frac{1}{6q^3} \right)^{3/4} 
		\]
		which implies that 
		\[
		{\rm mac}(D_n) \leq   \frac{|A(D_n)|}{4} +  \frac{|A(D_n)|^{3/4}}{\left( \frac{\alpha}{2} - \frac{\alpha}{2q} - \frac{1}{6q} + \frac{1}{6q^3} \right)^{3/4}} \times \frac{3\alpha +1}{24} .
		\]
		Since $\alpha = n/q^2$, we have 
		\[
		1 \leq \frac{n}{\lfloor \sqrt{n} \rfloor^2} = \alpha = \frac{(\sqrt{n})^2}{q^2} \leq \frac{(q+1)^2}{q^2} = \left( 1 + \frac{1}{q} \right)^2 \leq 2.25 .
		\]
		Therefore,
		\[
		\frac{3\alpha +1}{24\left( \alpha \left(\frac{1}{2} - \frac{1}{2q} \right) - \frac{1}{6q} + \frac{1}{6q^3} \right)^{3/4}} 
		\leq \frac{7.75}{24\left( \frac{1}{2} - \frac{1}{2q} - \frac{1}{6q}\right)^{3/4}}  
		\leq  \frac{7.75}{24\left( \frac{1}{2} - \frac{1}{4} - \frac{1}{12}\right)^{3/4}} .
		\]
		So, if we let $k'_1 = \frac{7.75}{24\left( \frac{1}{2} - \frac{1}{4} - \frac{1}{12}\right)^{3/4}}
		$, then 
		\[
		{\rm mac}(D_n) \leq   \frac{|A(D_n)|}{4} +  k'_1 |A(D_n)|^{3/4}.
		\]
		
		To complete the proof, we will show that the bound can be extended to every directed multigraph $D'_m$ with $n$ vertices and $m\ge 1$ arcs.
	       For any integer $n\ge 2$, define the function 
	       \[ f(n)=\frac{nq^2}{2} - \frac{nq}{2} - \frac{q^3}{6} + \frac{q}{6},\] where $q=\lfloor \sqrt{n} \rfloor.$ 
		We first prove that $f(n)$ is a monotonically increasing function. Let $n\ge 3.$ Suppose first that $\lfloor \sqrt{n} \rfloor=\lfloor \sqrt{n-1} \rfloor=q.$
		Then \[ f(n)-f(n-1)=q(q-1)/2\ge 0.\] Now suppose that $\lfloor \sqrt{n-1} \rfloor=q-1.$ Then
		\begin{align*}
			f(n)-f(n-1) & =  \frac{nq(q-1)}{2} - \frac{q^3}{6} + \frac{q}{6} - \left( \frac{(n-1)(q-1)(q-2)}{2}  - \frac{(q-1)^3}{6} + \frac{q-1}{6}  \right) \\ 
			& =   (q-1)(n-1)\ge 0.
		\end{align*}
		Since $q/2\leq n-1$, we always have $f(n)-f(n-1)\leq (q-1)(n-1)$. Let $n$ be the smallest integer such that $m \leq |A(D_n)| = f(n)$, i.e., 
		$f(n-1) < m \leq f(n).$
		Let $m_n = |A(D_n)|$. Let $D'_m$ be obtained from $D_n$ by deleting any $m_n-m$ arcs from $D_n$. Note that
		\[
			m_n-m =  f(n) - m 
		         \le  f(n) - f(n-1) 
			 \le  (q-1)(n-1)  
			 \le  nq 
			 \le   n^{3/2} .  
		\]
		Furthermore,
		\begin{align*}
			m & >  f(n-1)\\
			& \ge    \frac{(n-1)\lfloor \sqrt{n-1} \rfloor^2}{2} - \frac{(n-1)\lfloor \sqrt{n-1} \rfloor}{2} - \frac{\lfloor \sqrt{n-1} \rfloor^3}{6} + \frac{\lfloor \sqrt{n-1} \rfloor}{6} \\
			& >   \frac{(n-1)(\sqrt{n-1} -1)^2}{2} - \frac{(n-1)\sqrt{n-1}}{2} - \frac{\sqrt{n-1}^3}{6} \\
			& =  \frac{n^2}{2} - \frac{n}{2} - \frac{5(n-1)\sqrt{n-1}}{3} . \\
		\end{align*}
		
		We will show that $m > \frac{n^2}{144}$ if  $n \geq 12$. This follows from the above if we show that  for all $n \geq 12$,
		\[
		\frac{n^2}{2} - \frac{n}{2} - \frac{5(n-1)\sqrt{n-1}}{3}  >  \frac{n^2}{144}.
		\]
		This is equivalent to 		\[
		n^2 \left( 1 - \frac{1}{72} \right)  > n + \frac{10(n-1)\sqrt{n-1}}{3}.
		\]
		As $n \geq 12$, we have $(1-1/72) \cdot \sqrt{n} >3.41$, which implies
		\[
		n^{3/2} \cdot n^{1/2} \left( 1 - \frac{1}{72} \right) > 3.41 \cdot n^{3/2} >  n + \frac{10(n-1)\sqrt{n-1}}{3}.
		\]
		So  $m > \frac{n^2}{144}$ if $n \geq 12$. If $n <12$ we have $m \geq 1 > n^2/144$, so 
		$m >  \frac{n^2}{144}$ always holds. This implies that $m^{3/4} > \frac{n^{3/2}}{144^{3/4}}$ and therefore $42 m^{3/4} > 144^{3/4} m^{3/4} > n^{3/2}$.
		Now,
		\begin{align*}
			{\rm mac}(D'_m) & \leq  {\rm mac}(D_n)  \leq   \frac{|A(D_n)|}{4} +  k'_1 |A(D_n)|^{3/4} \\
			& =   \frac{m + (m_n-m)}{4} +  k'_1 (m + (m_n'-m))^{3/4} \\
			& <  \frac{m + n^{3/2}}{4} +  k'_1 (m + n^{3/2})^{3/4} \\
			& <  \frac{m + 42m^{3/4}}{4} +  k'_1 (m + 42m^{3/4})^{3/4} \\
			& <  \frac{m}{4} + 11m^{3/4} +  k'_1 (43m)^{3/4} \\
			& =  \frac{m}{4} + \left(11 + k'_1 \cdot 43^{3/4} \right) m^{3/4}.  \\
		\end{align*}
	So the required bound holds for all $m$ with $k_1 = 11 + k'_1 \cdot 43^{3/4}$.
	\end{proof}
		By monotonicity of $h_\omega (w)$, Theorem \ref{thm:upperB} implies the following corollary.
	\begin{corollary}
	$h_\omega (w)=\frac{w}{4}+O(w^{3/4})$. 
	\end{corollary}
	\begin{theorem}\label{thm:dag}
		For all acyclic digraphs $D\in {\cal D}_{\omega}(w)$, we have ${\rm mac}(D) \geq \frac{w(D)}{4} + \frac{ w(D)^{0.6}}{24}$.
	\end{theorem}	
	\begin{proof}
		Let $D$ be a weighted acyclic digraph in ${\cal D}_{\omega}(w)$ and let $P=p_1  p_2 \ldots p_l$ be a longest path in $D$.
		If $l \geq w(D)^{0.6}$ then consider the two matchings $M_0 = \{p_1p_2, p_3p_4, p_5p_6, \ldots \}$ and
		$M_1 = \{p_2p_3, p_4p_5, p_6p_7, \ldots \}$. Note that $w(M_0)+w(M_1)=w(P) \geq l-1 \geq w(D)^{0.6}-1$. 
		By Corollary~\ref{lowerBoundMatching} we obtain that ${\rm mac}(D) \geq \frac{w(D)}{4} + \frac{w(M_0)}{4}$ and
		${\rm mac}(D) \geq \frac{w(D)}{4} + \frac{w(M_1)}{4}$, which implies the following when $w(D) \geq 2$. 
		
		\[
		{\rm mac}(D) \geq \frac{w(D)}{4} + \frac{w(M_0)+w(M_1)}{8} \geq \frac{w(D)}{4} + \frac{w(D)^{0.6}-1}{8} \geq \frac{w(D)}{4} + \frac{w(D)^{0.6}}{24}
		.\]
		And if $w(D)<2$ then $|A(D)| \leq 1$ and ${\rm mac}(D) = w(D) \geq \frac{w(D)}{4} + \frac{w(D)^{0.6}}{24}$.
		So, we may assume that $l < w(D)^{0.6}$, which by Theorem~\ref{lower}
		implies\footnote{Note that  the proof of Theorem~\ref{lower} does not use Theorem \ref{thm:dag}, i.e., there is no circular dependency between the two theorems.} the following:
		\[
		{\rm mac}(D) \geq \frac{w(D)}{4} + \frac{k_2 \cdot w(D)}{l^{2/3}} 
		\geq \frac{w(D)}{4} + \frac{k_2 \cdot w(D)}{(w(D)^{0.6})^{2/3}}
		\geq \frac{w(D)}{4} + k_2 \cdot w(D)^{0.6}.
		\]
		Then we are done because $k_2>\frac{1}{24}$.
	\end{proof}
	
		Let $\alpha_0$ be the supremum of $\alpha>0$ such that there exists a  constant $k>0$ so that for all digraphs in the class below the following holds: 
		${\rm mac}(D) \geq \frac{w(D)}{4} + k w(D)^{\alpha}$. We have the following:
	% \multicolumn{4}{|c|}{The smallest known intervals containing $\alpha$} \\ \hline \hline
	% \multicolumn{2}{c}{Multi-column}
	
	\begin{center}
		\begin{tabular}{|c||c|c|c|} \hline
			\multicolumn{4}{|c|}{What is known about $\alpha_0$} \\  \hline \hline
			& & \multicolumn{2}{|c|}{Unweighted} \\
			& Weighted         & Directed  & Simple  \\ 
			& digraphs $D\in {\cal D}_{\omega}(w)$ & multigraphs & digraphs \\ \hline \hline
			Any digraph     & $\alpha_0=0.5$ & $\alpha_0=0.5$ & $\alpha_0=0.5$ \\ \hline
			Acyclic digraph & $\alpha_0 \in [0.6,0.75]$ & $\alpha_0 \in [0.6,0.75]$ & $ \alpha_0 \in [0.6,0.80]$\\ \hline
		\end{tabular}
	\end{center}
	Naturally, we have the following:
	\begin{openProblem} \label{probAC}
		What are the exact values of $\alpha_0$ in the above table?
	\end{openProblem} 	
	\section{Acyclic Digraphs with Bounded Path Length}\label{sec: acyclic digraphs with bounded path length}
	
	In this section we study the following problem.
		
		\begin{openProblem} \label{probACYCpathLenght}
			For each $\nu \ge 1$, determine the supremum $c_\nu$ of $c\ge 0$, such that all weighted acyclic digraphs with maximum path order  $\nu$ satisfy
			${\rm mac}(D) \geq c \cdot  w(D)$.
		\end{openProblem}
	\AZ{  Note that we do not impose constraints of the arc weights in the problem above.
		Note that $c_\nu$ is non-increasing. Indeed, for any digraph $D$ with the longest path order $\nu$ and $\mathrm{mac}(D)=c\cdot w(D)$, we can construct a digraph $D'$ from $D$ by adding a new vertex and a new arc with weight 0 to its longest path such that the longest path of $D'$ has $\nu+1$ vertices. Therefore, $\mathrm{mac}(D')=\mathrm{mac}(D)=c\cdot w(D)=c \cdot w(D')$. This implies $c_{\nu+1}\leq c_\nu$.
      }

	We can compute the exact values of $c_\nu$ when $\nu\leq 11$ (see Appendix). In this section, we focus on bounds for $c_\nu$. Since the number of vertices in the longest path of $D_n$ in Theorem \ref{thm:upperB} is $n$, we can easily obtain the following upper bound by considering $D_\nu$. 
		
		\begin{theorem}\label{thm:upperBr}
			For $\nu\geq 12$, we have $c_\nu \leq \frac{1}{4} + \frac{1}{3\sqrt{\nu}-10}$.
		\end{theorem}
\begin{proof}
		For each $\nu\geq 12$, let $D_\nu$ be defined as in Theorem \ref{thm:upperB}. Recall that
		
		\[
		\begin{array}{rcl} \vspace{0.1cm}
			{\rm mac}(D_\nu) &\leq& \frac{|A(D_\nu)|}{4} + \frac{3\nu\lfloor \sqrt{\nu} \rfloor+\lfloor \sqrt{\nu} \rfloor^3}{24}\\\vspace{0.1cm}
			&=& |A(D_\nu)|\left(  \frac{1}{4} +  \frac{3\nu\lfloor \sqrt{\nu} \rfloor+\lfloor \sqrt{\nu} \rfloor^3}{24\left( \frac{\nu\lfloor \sqrt{\nu} \rfloor^2}{2} - \frac{\nu\lfloor \sqrt{\nu} \rfloor}{2} - \frac{\lfloor \sqrt{\nu} \rfloor^3}{6} + \frac{\lfloor \sqrt{\nu} \rfloor}{6} \right)} \right). \\ 
		\end{array}
		\]
		By using $\sqrt{\nu}-1 \leq \lfloor \sqrt{\nu}\rfloor\leq \sqrt{\nu}$, we have
		\[
		\begin{array}{rcl} \vspace{0.1cm}
			{\rm mac}(D_\nu) &\leq & |A(D_\nu)|\left(  \frac{1}{4} +  \frac{\nu^{{3/2}}}{3\nu(\sqrt{\nu}-1)^2 - 3\nu^{{3/2}}- \nu^{{3/2}} + \sqrt{\nu}-1} \right)\\ \vspace{0.1cm}
			& \leq & |A(D_\nu)|\left(  \frac{1}{4} +  \frac{\nu^{{3/2}}}{3\nu^2 - 10\nu^{{3/2}} } \right)\\ \vspace{0.1cm}
			& = & |A(D_\nu)|\left(  \frac{1}{4} +  \frac{1}{3\sqrt{\nu} - 10} \right).
		\end{array}
		\]
		Therefore, $c_{\nu}\leq   \frac{1}{4} +  \frac{1}{3\sqrt{\nu} - 10} $, which completes the proof.
	\end{proof}

		The main goal of this section is to prove a lower bound for $c_\nu$. The main idea of how we obtain the lower bound for all $\nu$ is roughly the following: we first prove (in Theorem \ref{bigRiii}) that $c_\nu= \frac{1}{4}+\Omega(\nu^{-2/3})$ for every $\nu \in N^*$,
		where $N^*=\{n_i^*:\ i\in \mathbb{N}, i\ge 2\}$ (the elements of $N^*$ will be defined later).
Thus, for any $\nu\ge n_2^*$, we may assume $n_{k-1}^*\leq \nu< n_{k}^*$ for some $k$. And we will show that the gap between these two numbers is not too large compared to $\nu$ (namely, $n_{k}^*-n_{k-1}^*=O(\nu^{2/3})$). As $c_\nu$ is non-increasing, we conclude that $c_\nu\geq c_{n^*_k}= \frac{1}{4}+\Omega({n^*_k}^{-2/3})=\frac{1}{4}+\Omega({(\nu+O(\nu^{2/3}))}^{-2/3})=\frac{1}{4}+\Omega({\nu}^{-2/3})$. For each integer $k\geq 2$, let
	\begin{align*}
		n_k^* & =   2k + 2\sum_{i=1}^{\lfloor \sqrt{k/2} \rfloor} \left(2\left\lfloor \frac{2k-2i^2-i}{2} \right\rfloor+i \right) \\ 
		& =   2k + 4k \lfloor \sqrt{k/2} \rfloor - 2\left\lfloor \frac{\sqrt{k/2}+1}{2} \right\rfloor - 4\sum_{i=1}^{\lfloor \sqrt{k/2} \rfloor} i^2 \\
		& =   2k + 4k \lfloor \sqrt{k/2} \rfloor - 2\left\lfloor \frac{\sqrt{k/2}+1}{2} \right\rfloor - \frac{2}{3} \lfloor \sqrt{k/2} \rfloor (\lfloor \sqrt{k/2} \rfloor+1)(2\lfloor \sqrt{k/2} \rfloor+1) . \\
	\end{align*}

	\begin{theorem} \label{bigRiii}
		For any integer $k\geq 7$, we have $c_{n_k^*} \geq \frac{1}{4}+\frac{1}{8{n_k^*}^{2/3}-4}$.
	\end{theorem}
	\begin{proof}
		For any positive integer $k\geq 7$ and $0 \leq q \leq z=\lfloor \sqrt{k/2} \rfloor$,
		let $f(q)= \left\lfloor \frac{2k-2q^2-q}{2} \right\rfloor$. Let $S_q$ be any set of size $2f(q) + q$. 
		
		\2
		
		\noindent{\bf Claim A:} Let $(X,Y)$ be a random partition of $S_q$, with $|X|=f(q)+q$ and $|Y|=f(q)$. 
		For any distinct $s_1,s_2 \in S_q$, the probability that $s_1 \in X$ and $s_2 \in Y$ is at least $\frac{k}{4k-2}$.
		
		\2

		{\bf Proof of Claim A:} The probability that $s_1 \in X$ is $\frac{f(q)+q}{2f(q)+q}$. 
		And given that $s_1 \in X$, the probability that $s_2 \in Y$ is $\frac{f(q)}{2f(q)+q-1}$. 
		So the probability that $s_1 \in X$ and $s_2 \in Y$ is the following:
		
		\[
		\mathbb{P}(s_1 \in X, \;  s_2 \in Y) = \frac{f(q)+q}{2f(q)+q} \times \frac{f(q)}{2f(q)+q-1}.
		\]
		
		We first consider the case when $q$ is even in which case $f(q)= k-q^2-\frac{q}{2}$.
		In this case, $\mathbb{P}(s_1 \in X,  \; s_2 \in Y) \geq \frac{k}{4k-2}$ is equivalent to the 
		following inequality.
		
		\[
		\frac{(k-q^2+q/2)(k-q^2-q/2)}{(2k-2q^2)(2k-2q^2-1)} \geq \frac{k}{4k-2}.
		\]
		
		This is equivalent to the following:
		
		\[
		(4k-2) \left[(k-q^2)^2 - \frac{q^2}{4}\right]  \geq 4k(k-q^2) \left(k-q^2- \frac{1}{2} \right).
		\]
		
		Subtracting $4k(k-q^2)^2$ from both sides gives us the following equivalent inequality.
		
		\[
		(-2) \cdot \left[(k-q^2)^2 - \frac{q^2}{4}\right] - kq^2  \geq -2k(k-q^2) .
		\]
		
		%Moving everything to the right and multiplying out the brackets gives us the following equivalent inequality.
		Thus,
		
		\[
		0 \geq 2k^2 -4kq^2 + 2q^4 - \frac{q^2}{2} + kq^2 -2k(k-q^2) .
		\]
		
		We note that this is equivalent to $0 \geq 2q^4 - kq^2 - \frac{q^2}{2}$.  We recall that $q \leq \lfloor \sqrt{k/2} \rfloor$, which implies that 
		$kq^2 \geq 2q^4$, which in turn implies that $0 \geq 2q^4 - kq^2 - \frac{q^2}{2}$ holds.  Therefore,
		$\mathbb{P}(s_1 \in X, \; s_2 \in Y) \geq \frac{k}{4k-2}$ also holds.
		
		We now consider the case when $q$ is odd, in which case $f(q)= k-q^2-\frac{q}{2}-0.5$.
		Let $k^* = k - 1/2$ and note that the following holds. 
		
		\begin{align*}
			\mathbb{P}(s_1 \in X, \;  s_2 \in Y) & =  \frac{f(q)+q}{2f(q)+q} \times \frac{f(q)}{2f(q)+q-1} \\ 
						& =  \frac{(k-q^2+q/2-0.5)(2k-q^2-q/2-0.5)}{(2k-2q^2-1)(2k-2q^2-1-1)} \\
			& =  \frac{(k^*-q^2+q/2)(k^*-q^2-q/2)}{(2k^*-2q^2)(2k^*-2q^2-1)} .\\
		\end{align*}
		
		Using the same computations as above (but with $k^*$ instead of $k$) we note that $\mathbb{P}(s_1 \in X, \;  s_2 \in Y) \geq
		\frac{k^*}{4k^*-2}$ is equivalent with $0 \geq 2q^4 - k^*q^2 - \frac{q^2}{2}$.
		As $kq^2 \geq 2q^4$ we have $2q^4 \leq k q^2 = k^* q^2 + q^2/2$, which implies that $0 \geq 2q^4 - k^*q^2 - \frac{q^2}{2}$ holds.
		Therefore the following holds.
		
		\[
		\mathbb{P}(s_1 \in X, \;  s_2 \in Y) \geq \frac{k^*}{4k^*-2} = \frac{k-1/2}{4k-4} \geq \frac{k}{4k-2}.
		\]
		
		This completes the proof of Claim~A. \smallQED{}
		
		\2
		
		Let $n_q = 2f(q)+q$ and
		let $D$ be any acyclic digraph whose longest path has order $n_k^*$. Recall that $n^*_k$ is
		
		\[
		n_k^* = 2k + 2\sum_{i=1}^{\lfloor \sqrt{k/2} \rfloor} (2f(i)+i) = 2k + 2\sum_{i=1}^{\lfloor \sqrt{k/2} \rfloor} n_i.
		\]
		
		\noindent{\bf Claim B:} $n_k^* \geq k^{3/2}$.
		
		\2

		{\bf Proof of Claim B:} Let $z=\lfloor \sqrt{k/2} \rfloor$. As $\frac{5\sqrt{2}}{3} - 3k^{-1/2} - \frac{5\sqrt{2}}{6}k^{-1} - k^{-3/2}\geq 1$ for all $k \geq 7$ the following holds.

		\begin{align*}
			n_k^* & =  2k + 2\sum_{i=1}^z (2f(i)+i) \\ 
			& =  2k  + 2\sum_{i=1}^z \left(  2  \left\lfloor \frac{2k-2i^2-i}{2} \right\rfloor + i \right) \\ 
			& =  2k - 2\left\lfloor \frac{z+1}{2} \right\rfloor + 2\sum_{i=1}^z \left(  2 \left( \frac{2k-2i^2-i}{2} \right) + i \right) \\ 
			& =  2k - 2\left\lfloor \frac{z+1}{2} \right\rfloor + 4kz - 4 \sum_{i=1}^z i^2 \\ 
			%& = & 2k + 4kz - 4 \times \frac{z(z+1)(2z+1)}{6} \\ \vspace{0.1cm}
			& =  2k - 2\left\lfloor \frac{z+1}{2} \right\rfloor  + 4kz - 4 \times \frac{2z^3+3z^2+z}{6} \\ 
			& =  2k - 2\left\lfloor \frac{z+1}{2} \right\rfloor  + 4kz - \frac{4z^3+6z^2+2z}{3} .
		\end{align*}
		By using $\sqrt{x}-1\leq \lfloor \sqrt{x} \rfloor\leq \sqrt{x}$, we have
		\begin{align*}
			n_k^* & \geq  2k - 2\left(\frac{(k/2)^{1/2}+1}{2}\right)  + 4k((k/2)^{1/2}-1) - \frac{\sqrt{2}k^{3/2}+3k+\sqrt{2}k^{1/2}}{3} \\ \vspace{0.1cm}
			& \geq  \frac{5\sqrt{2}}{3}k^{3/2} - 3k - ((k/2)^{1/2}+1) -  \frac{\sqrt{2}k^{1/2}}{3} \\ 
			& =  \frac{5\sqrt{2}}{3}k^{3/2} - 3k - \frac{5\sqrt{2}}{6}k^{1/2} - 1 \\
			& = 
			k^{3/2}\left(\frac{5\sqrt{2}}{3} - 3k^{-1/2} - \frac{5\sqrt{2}}{6}k^{-1} - k^{-3/2}\right)  \\
			&\geq  k^{3/2}. \\
		\end{align*}
		
		This completes the proof of Claim~B. \smallQED{}
		
		\2 
		%\AZ{Let D be a weigthed acyclic digraph with maximum path order $\nu$.}
		Let $T_1$ contain all sources in $D$ and let
		$T_2$ contain all sources in $D-T_1$. 
		Continuing in this way we obtain the sets $T_1,T_2,\ldots,T_{n_k^*}$. 
		Contract each $T_i$ into the vertex $v_i$ and let $D'$ denote the resulting digraph.
		Note that $D'$ is acyclic and 
		$(v_1,v_2,\ldots,v_{n_k^*})$ is an acyclic ordering of $D'$ (i.e. if $v_i v_j \in A(D')$ then $i<j$).
		The weight of an arc $v_i v_j$ is just the sum of the weights of all arcs from $T_i$ to $T_j$ in $D'$.
		
		Let $A_{-z}$ denote the first $n_z$ vertices in the ordering. Thus, $$A_{-z}=\{v_1,v_2,\ldots,v_{n_z}\}.$$
		Let $A_{-(z-1)}$ denote the following $n_{z-1}$ vertices and we continue this process until $A_{-1}$ is defined.
		%Thus, 
		%$$A_{-(z-1)}=\{v_{n_z+1},v_{n_z+2},\ldots,v_{n_z+n_{z-1}}\}.$$
		%Let $A_{-(z-2)}$ denote the following $n_{z-2}$ vertices. 
		Now let $A_0$ denote the following $2k$ vertices in the ordering.
		Let $A_{1}$ denote the following $n_1$ vertices and 
		let $A_{2}$ denote the following $n_2$ vertices.
		Continue this process until $A_z$ is defined and note that $A_z$ defines the last $n_z$ vertices in the acyclic ordering
		(i.e. $A_z=\{v_{n_k^*-n_z+1},v_{n_k^*-n_z+2},\ldots,v_{n_k^*}\}$).
		
		We now produce a "random" solution as follows. Pick a partition $(X_i,Y_i)$ of $A_i$ at random such that
		$|X_i|=f(|i|)+|i|$ and $|Y_i|=f(|i|)$ for each $-z \leq i \leq 0$ and
		pick a partition $(X_i,Y_i)$ of $A_i$ at random such that
		$|X_i|=f(i)$ and $|Y_i|=f(i)+i$ for each $1 \leq i \leq z$.
		Let $X^* = \cup_{i=-z}^z X_i$ and $Y^* = \cup_{i=-z}^z Y_i$
		
		We will in Claim~C show that every arc of $D'$ has probability at least $k/(4k-2)$ of belonging to the dicut $(X^*,Y^*)$.
		
		\2
		
		\noindent{\bf Claim C:} $\mathbb{E} (w(X^*,Y^*)) \geq w(D') \times \frac{k}{4k-2}$.
		
		\2
		
		{\bf Proof of Claim C:}  Let $uv \in A(D')$ be an arbitrary arc. 
		If $uv \in A_i$ for some $i$ then $uv$ belongs to the dicut $(X^*,Y^*)$ with probability 
		at least $\frac{k}{4k-2}$ by Claim~A. So assume that $u \in A_s$ and $v \in A_t$ where $s<t$.

		Let $w_i \in A_i$ be arbitrary and note that the probability that $w_i \in X^*$ is the following.
		
		\[
		\mathbb{P}(w_i \in X^*) = \left\{ \begin{array}{lcl} \vspace{0.12cm}
			\frac{f(|i|)+|i|}{2f(|i|)+|i|}=\frac{1}{2}+\frac{1}{4f(|i|)/|i|+2} & \hspace{1cm} & \mbox{if $i < 0$}, \\
			\frac{f(i)}{2f(i)+i}= \frac{1}{2}-\frac{1}{4f(i)/i+2} & \hspace{1cm} & \mbox{if $i \geq 0$}. \\
		\end{array} \right.
		\]
		When $i\geq 0$, since $f(i)$ is decreasing, $f(i)/i$ is clearly decreasing and therefore $\mathbb{P}(w_i \in X^*)$ is decreasing. When $i<0$, $f(|i|)/|i|=f(-i)/-i$ is increasing which implies $\mathbb{P}(w_i \in X^*)$ is also decreasing. These together with the fact that $\mathbb{P}(w_i \in X^*)> 1/2$ when $i< 0$ and $\mathbb{P}(w_i \in X^*)\leq 1/2$ when $i\geq 0$ imply that
		$\mathbb{P}(w_i \in X^*) > \mathbb{P}(w_j \in X^*)$ if and only if $i<j$.
		Analogously, $\mathbb{P}(w_i \in Y^*) < \mathbb{P}(w_j \in Y^*)$ if and only if $i<j$.
		
		We now consider the following two cases, which exhaust all possibilities, as $s<t$.
		
		\2
		
		\noindent{\bf Case C.1. $t>0$:}  Note that $\mathbb{P}(u \in X^*) \geq \mathbb{P}(w_{t-1} \in X^*)$, as $s \leq t-1$.
		Therefore the following holds (where we consider the case when $v \in Y^*$ is given), where $q_t$ is an arbitrary vertex in $A_t \setminus v$.
		
		\begin{align*}
			\mathbb{P}(u \in X^* \; | \; v \in Y^*) & =  \mathbb{P}(u \in X^*) \\
			& \geq  \mathbb{P}(w_{t-1} \in X^*) \\  
			& =  \frac{f(t-1)}{2f(t-1)+(t-1)}. \\
		\end{align*}
		
		As $f(t-1) \geq f(t)$ we note that $\frac{f(t-1)}{2f(t-1)+(t-1)} \geq \frac{f(t)}{2f(t)+(t-1)}$, which implies the following 
		(by the above).
		
		\[
		\mathbb{P}(u \in X^* \; | \; v \in Y^*)  \geq  \frac{f(t)}{2f(t)+(t-1)} 
		=  \mathbb{P}(q_t \in X^* \; | \; v \in Y^*).
		\]
		
		Therefore the probability of $uv$ belonging to the partition $(X^*,Y^*)$ is at least as great as the probability 
		that $q_t v$ belonging to the partition $(X^*,Y^*)$. Therefore, as Claim~A implies that 
		$\mathbb{P}(q_t v \in (X^*,Y^*)) \geq \frac{k}{4k-2}$, we have completed Case~C.1.

		\2

		\noindent{\bf Case C.2. $s<0$:} This case can be proved analogously to Case~C.1, by letting $q_s \in A_s \setminus \{u\}$ be arbitrary 
		and showing the following.
		
		\[
		\mathbb{P}(v \in Y^* \; | \; u \in X^*) \geq \mathbb{P}(q_s \in Y^* \; | \; u \in X^*).
		\]
		
		And then using Claim~A to show that $\mathbb{P}(u q_s \in (X^*,Y^*)) \geq \frac{k}{4k-2}$, 
		which completes Case~C.2.\smallQED{}
		
		\2

		From Claim~B we have $n_k^* \geq k^{3/2}$, which implies that $k \leq {n_k^*}^{2/3}$.
		By Claim~C, 
		
		\[
		\mathbb{E} (w(X^*,Y^*)) \geq w(D') \times \frac{k}{4k-2} \geq  w(D) \times \frac{{n_k^*}^{2/3}}{4{n_k^*}^{2/3}-2},
		\]	
		which completes the proof. \end{proof}
	\begin{theorem}\label{lower}
		Let $k_2=\frac{1}{8\times 3^{2/3}}$.	Then for any $\nu\geq 1$, we have $c_\nu \geq \frac{1}{4}+k_2\nu^{-2/3}$.
	\end{theorem}
	\begin{proof}
		Note that when $\nu<36$ we are done by the fact that $\frac{1}{4}+\frac{1}{4\nu}> \frac{1}{4}+k_2\nu^{-2/3}$ and Theorem \ref{cut(r)}. So we assume $\nu\geq n^*_{7}=36$.
		Suppose without loss of generality that $n^*_{k-1}\leq\nu\leq n^*_k$. Recall the definition of $n_k^*$:
		\begin{align*}
			n_k^* & =   2k + 2\sum_{i=1}^{\lfloor \sqrt{k/2} \rfloor} \left(2\left\lfloor \frac{2k-2i^2-i}{2} \right\rfloor+i \right) \\ 
			& =   2k + 4k \lfloor \sqrt{k/2} \rfloor - 2\left\lfloor \frac{\sqrt{k/2}+1}{2} \right\rfloor - 4\sum_{i=1}^{\lfloor \sqrt{k/2} \rfloor} i^2 
		\end{align*}
		
		For $k\geq 7$, we have 
		\begin{align*}
			n_k^*-n_{k-1}^*&\leq 2+4k(\lfloor\sqrt{k/2}\rfloor-\lfloor\sqrt{(k-1)/2}\rfloor)+4\lfloor\sqrt{(k-1)/2}\rfloor\\
			&\leq 2+4k+2\sqrt{2(k-1)}\\
			&\leq 6k.
		\end{align*}

		The last inequalty follows from the fact that $k\geq 7$. By Claim B in Theorem \ref{bigRiii}, $\nu\geq n^*_{k-1}\geq (k-1)^{3/2}$, $k\leq \nu^{2/3}+1$. Thus,
		$n_k^*-\nu\leq n_k^*-n_{k-1}^*$ and $n_k^*\leq n_k^*-n^*_{k-1}+\nu\leq \nu+6\nu^{2/3}+6$. Since $c_\nu$ is \AZ{non-increasing}, we have
		\begin{align*}
			c_\nu &\geq c_{n^*_k} \\
			&\geq  \frac{1}{4}+\frac{1}{8{n_k^*}^{2/3}-4}\\
			&\geq  \frac{1}{4}+\frac{1}{8(\nu+6+6\nu^{2/3})^{2/3}-4}\\
			&\geq  \frac{1}{4}+\frac{1}{8\times 3^{2/3}\nu^{2/3}}\\
			&=  \frac{1}{4}+k_2\nu^{-2/3},
		\end{align*}
				where the last inequality follows from the fact that $\nu\geq 36$. This completes the proof.
	\end{proof}

	We complete this section with the following:	

		\begin{openProblem}
			It'd be interesting to identify the  infimum $\alpha_U$ of $\alpha>0$ such that
			\(c_\nu= 1/4 + O(\nu^\alpha)\) and the supremum $\alpha_L$ of $\alpha>0$ such that
			\(c_\nu= 1/4 + \Omega(\nu^\alpha).\) Is $\alpha_U=\alpha_L$ ?
		\end{openProblem}
		We already know from the upper bound and the lower bound that $-2/3\le \alpha_L\le \alpha_U\le -1/2$.

	\section{Generalization of Theorem \ref{thm:dag}}\label{Generalization of Theorem}
\AZ{In this section, we generalize the lower bound for weighted acyclic digraphs in Theorem \ref{thm:dag} to weighted digraphs in ${\cal D}_{\omega}(w)$ with the length of every circle bounded by a constant. Our proof will use the below theorem of Bondy.} Let $\circ(D)$ denote the {\em circumference} of a digraph $D$ i.e. the \AZ{length of a longest  cycle in $D$. The {\em chromatic number} $\chi(D)$ of a digraph $D$ is the chromatic number of the underlying graph of $D.$
}

\begin{theorem} \label{ThmBondy} \cite{Bondy}
	For all strong digraphs $D$ we have $\chi(D) \leq \circ(D)$.
\end{theorem}

\begin{theorem} \label{ThmMain}
	Assume that there exist constants $k>0$ and $0<\alpha<1$ such that ${\rm mac}(H) \geq \frac{w_H(H)}{4} + k w_H(H)^{\alpha}$ for all acyclic digraphs $H=(V(H),A(H),w_H)\in {\cal D}_{\omega}(w)$.
	Let \AZ{$D=(V(D),A(D),w)$} be an arbitrary digraph \AZ{in ${\cal D}_{\omega}(w)$} and let $A_s$ \AZ{consist of} all arcs of $D$ contained within strong components of $D$ and let $A_a=A(D) \setminus A_s$.
	Then the following holds.
	
	\[
	{\rm mac}(D) \geq \frac{w(D)}{4} + \frac{k}{(4k+1) \cdot \circ(D)+1} \times (w(D[A_a])^{\alpha} + w(D[A_s])).
	\]
\end{theorem}

\begin{proof}
	Let $D$, $A_s$ and $A_a$ be defined as in the statement of the theorem.
	Let $S_1,S_2,\ldots,S_r$ denote the strong components of $D$ and let
	$D_a$ denote the acyclic digraph obtained from $D$ by contracting each $S_i$ into the vertex $u_i$ (note that $S_i$ may have order one).
	Furthermore, if there are multiple arcs from some $u_i$ to $u_j$ then replace these by a single arc with weight equal to the sum of the weights of all the multiple
	arcs.  Let $w_a$ denote this new weight function in $D_a$. Note that $w_a(D_a) = w(\AZ{D[A_a]})$ and thus
		\begin{align}  \label{gen1}
		  \Mac(D) &\geq {\Mac}(D_a)  \geq  \frac{w_a(D_a)}{4} + k \cdot w_a(D_a)^{\alpha} \\ \nonumber
		& =  \frac{w(D)-w(\AZ{D[A_s]})}{4} + k \cdot w(\AZ{D[A_a]})^{\alpha}. 
\end{align}
	
	We now prove an alternative bound using Theorem~\ref{ThmBondy}. For each strong component $S_i$ in $D$ we know that $\chi(S_i) \leq \circ(D)$ by 
	Theorem~\ref{ThmBondy}. Let $l=\circ(D)$ and let $C_1^i, C_2^i, \ldots, C_l^i$ be a partition of $V(S_i)$ such that each $C_j^i$ is an independent set in $D$ (some
	$C_j^i$ may be empty). Following a similar approach to that  in the proof of Theorem~\ref{thm:basic} (d) we construct a cut $(X,Y)$, by splitting $C_1^i, C_2^i, \ldots, C_l^i$ into two sets uniformly at random over all splits that differ by at most one. One part will be part of $X$ and the other part of $Y$.
	%Let $r=(r_1,r_2,\ldots,r_l)$ be a uniformly randomly chosen bit-string of length $l$ where the number of ones and number of zeros differ by at most one.
	%Then for each $j=1,2,\ldots,l$ assign all the vertices of $C_j^i$ to $X$ if $r_j=0$ and to $Y$ if $r_j=1$.
	Note that each arc of $S_i$ is an $(X,Y)$-arc with probability $\frac{1}{4} \times \frac{l+1}{l}$ if $l$ is odd 
	and $\frac{1}{4} \times \frac{l}{l-1}$ if $l$ is even. Furthermore for all arcs in $A_a$ it will be an $(X,Y)$-arc with 
	probability $1/4$ as for $I \not= I'$,  $C_j^I$ and $C_{j'}^{I'}$ are in $X$ with probability $1/2$ independently of each other. 
	%\AZ{In fact, if $r=(r_1,r_2,\ldots,r_l)$ and $l=2k+1$ then for each $i\in [l]$ the probability of $r_i=0$ is $\frac{{2k \choose k}+{2k \choose k-1}}{2{2k+1 \choose k}}=1/2$ by Pascal's formula (similarly, one can show or observe that this is also true when $l$ is even).} 
	Therefore,
	
	\begin{equation} \label{gen2} {\rm mac}(D) \geq  \frac{w(D)}{4} + \left( \frac{l+1}{4l} - \frac{1}{4} \right) w(D[A_s]) =  \frac{w(D)}{4} + \frac{w(D[A_s])}{4l}. 
	\end{equation} 
	
	Adding inequality \eqref{gen1} together with $4lk+l$ times inequality \eqref{gen2}, gives
	\[
	(4lk+l+1){\rm mac}(D) \geq  \frac{(4lk+l+1) \cdot w(D)}{4} + k \cdot w(D[A_a])^{\alpha} + k \cdot w(D[A_s]).
	\]
	This implies  
	\[
	{\rm mac}(D) \geq  \frac{w(D)}{4} + \frac{k}{(4k+1) \cdot l + 1} (w(D[A_a])^{\alpha} + w(D[A_s]) ),
	\]
	which completes the proof of the theorem.
\end{proof}

\begin{corollary} \label{CorMain}
	Assume that there exist constants $k>0$ and $0<\alpha<1$ such that ${\rm mac}(H) \geq \frac{w_H(H)}{4} + k w_H(H)^{\alpha}$ for all  acyclic digraphs \AZ{$H\in {\cal D}_{\omega}(w)$}.
	Let \AZ{$D=(V(D),A(D),w)$} be an arbitrary digraph in ${\cal D}_{\omega}(w)$ with $\circ(D) \leq l$ for some fixed $l$. 
	Then the following holds.
	
	\[
	{\rm mac}(D) \geq \frac{w(D)}{4} + \frac{k}{(4k+1) \cdot l+1} \times w(D)^{\alpha}.
	\]
\end{corollary}

\begin{proof}
	Let $A_s$ contain all arcs of $D$ contained within strong components of $D$.
	The corollary holds as Theorem~\ref{ThmMain} implies the following (as $w(D) \geq 1$ and $0< \alpha < 1$).
	
	\[
	\begin{array}{rcl}
		{\rm mac}(D) & \geq & \frac{w(D)}{4} + \frac{k}{(4k+1) \cdot circ(D)+1} \times (w(D-A_s)^{\alpha} + w(D[A_s])) \\
		& \geq & \frac{w(D)}{4} + \frac{k}{(4k+1) \cdot l+1} \times (w(D)^{\alpha} - w(D[A_s])^{\alpha} + w(D[A_s])) \\
		& \geq & \frac{w(D)}{4} + \frac{k}{(4k+1) \cdot l+1} \times w(D)^{\alpha}, \\
	\end{array}
	\]
	which completes the proof.
\end{proof}
\AZ{
	By using the above result and Theorem \ref{thm:dag}, we have the following:
\begin{theorem}
	For any integer $l>0$. There exists a constant $k(l)>0$ such that the following holds.
	For all arc-weighted digraphs $D$ where each arc has weight at least one and $\circ(D)\leq l$ we have ${\rm mac}(D) \geq \frac{w(D)}{4} + k(l) \cdot w(D)^{0.6}$.
\end{theorem}
}

	\newpage
	\appendix
	
	{\bf \huge Appendix} 
%	\section{Exact Values of $c_{\nu}$ for Small $\nu$}\label{apdx: small nu}
	%Let $D$ be an acyclic digraph whose longest path has order \AZr{at most} $\nu$.
	 
\vspace{4mm}
	
	The following holds for small $\nu$.
	
\vspace{3mm}
	
\noindent	{\bf Theorem A}	 {\em		$c_2 = 1$, $c_3=c_4=\frac{1}{2}$, $c_5=c_6 = \frac{2}{5}$, $c_7=\frac{3}{8}$ and $c_8=\frac{4}{11}$.}

	\begin{proof}
		Let $D$ be an acyclic digraph whose longest path has order $\nu$, where $\nu \in \{2,3,4,5,6,7,8\}$.
		Let $S_1$ denote all sources in $D$. %Let $D_2 = D - S_1$ and let $S_2$ be all sources in $D_2$.
		For all $i=2,3,\ldots,\nu$ let $D_i = D_{i-1} - S_{i-1}$ and let $S_i$ be all sources in $D_i$.
		
		Note that $V(D) = S_1 \cup \cdots \cup S_\nu$ and every arc in $D$ is a $(S_i,S_j)$-arc for some $i<j$.
		Let $D'$ be obtained from $D$ by contracting each $S_i$ into a vertex $s_i$ for all $i=1,\ldots,\nu$.
		Let the weight of an arc $s_i s_j$ in $D'$ be $w(s_i s_j) = w(S_i,S_j)$ (i.e. the sum of the weights
		of all $(S_i,S_j)$-arcs).
		
		Note that ${\rm mac}(D) \geq {\rm mac}(D')$ as any dicut $(X',Y')$ in $D'$ can be made into a dicut of the same weight in $D$ by expanding each $s_i$ to $S_i$.
		Let $l_2 = 1$, $l_3=l_4=\frac{1}{2}$, $l_5=l_6=\frac{2}{5}$, $l_7=\frac{3}{8}$ and $l_8=\frac{4}{11}$. Clearly $c_2=1$, as dicut $(\{s_1\},\{s_2\})$ contains all arcs in $D'$ (as $A(D')=\{s_1s_2\}$), and therefore ${\rm mac}(D') = w(D')$.
		In Figure~\ref{fig:upperBcr} we see examples of acyclic digraphs, $D_\nu$, with maximum path order $\nu$, where ${\rm mac}(D_\nu) = l_\nu \cdot  w(D_\nu)$ for $\nu=3,4,5,6,7,8$ so 
		we only need to show that ${\rm mac}(D') \geq l_\nu \cdot w(D')$ for all $\nu=3,4,5,6,7,8$ to complete the proof.
		
		If $\nu=3$ then consider the dicuts $C_1^3=(\{s_1,s_2\},\{s_3\})$ and  $C_2^3=(\{s_1\},\{s_2,s_3\})$. Each arc in $A(D')$ belongs to at least one of the 
		dicuts, so ${\rm mac}(D') \geq w(D')/2$.
		
		If $\nu=4$ then consider the dicuts $C_1^4=(\{s_1,s_2\},\{s_3,s_4\})$ and  $C_2^4=(\{s_1,s_3\},\{s_2,s_4\})$. As can be seen in 
		the table below, every arc in $A(D')$ belongs to at least one of the  dicuts, so ${\rm mac}(D') \geq w(D')/2$.
		
		\begin{center}
			\begin{tabular}{|c|c|c||c|c|c|c|c|c|} \hline
				
				\multicolumn{3}{|c||}{Dicut $C_i^4=(X,Y)$} & \multicolumn{6}{|c|}{Contains the following arcs} \\
				$i$ & $X$       & $Y$       & $s_1s_2$ & $s_2s_3$ & $s_3s_4$ & $s_1s_3$ & $s_2s_4$ &  $s_1s_4$ \\ \hline \hline
				1   & $s_1,s_2$ & $s_3,s_4$ &          &    +     &          &    +     &    +     &     +     \\ \hline
				2   & $s_1,s_3$ & $s_2,s_4$ &    +     &          &    +     &          &          &     +     \\ \hline
			\end{tabular}
		\end{center}
		
		If $\nu=5$ then consider the dicuts $C_1^5=(\{s_1,s_2,s_3\},\{s_4,s_5\})$,$C_2^5=(\{s_1,s_2\}, \{s_3,s_4,s_5\})$, $C_3^5=(\{s_1,s_3,s_4\},\{s_2,s_5\})$, $C_4^5=(\{s_1,s_3\},\{s_2,s_4,s_5\})$ and
		$C_5^5=(\{s_1,s_2,s_4\},\{s_3,s_5\})$. 
		As can be seen in
		the below table every arc in $A(D')$ belongs to at least two of the five dicuts, so ${\rm mac}(D') \geq 2w(D')/5$.
		
		\begin{center}
			\begin{tabular}{|c|c|c||c|c|c|c|c|c|c|c|c|c|} \hline
				\multicolumn{3}{|c||}{Dicut $C_i^5=(X,Y)$} & \multicolumn{10}{|c|}{Contains the following arcs} \\
				$i$ & $X$           & $Y$           & \sX{$s_1s_2$} & \sX{$s_2s_3$} & \sX{$s_3s_4$} & \sX{$s_4s_5$} & \sX{$s_1s_3$} &  \sX{$s_2s_4$} & \sX{$s_3s_5$} & \sX{$s_1s_4$} & \sX{$s_2s_5$} & \sX{$s_1s_5$}  \\ \hline \hline
				1   & $s_1,s_2,s_3$ & $s_4,s_5$     &               &               &      +        &               &               &          +     &        +      &        +      &        +      &        +       \\ \hline
				2   & $s_1,s_2$     & $s_3,s_4,s_5$ &               &       +       &               &               &       +       &          +     &               &        +      &        +      &        +       \\ \hline
				3   & $s_1,s_3,s_4$ & $s_2,s_5$     &      +        &               &               &      +        &               &                &        +      &               &               &        +       \\ \hline
				4   & $s_1,s_3$     & $s_2,s_4,s_5$ &      +        &               &      +        &               &               &                &        +      &        +      &               &        +       \\ \hline
				5   & $s_1,s_2,s_4$ & $s_3,s_5$     &               &       +       &               &      +        &        +      &                &               &               &        +      &        +       \\ \hline
			\end{tabular}
		\end{center}
		
		If $\nu=6$ then consider the dicuts 
		$C_1^6=(\{s_1,s_2,        s_5\},\{s_3,s_4,s_6\})$,
		$C_2^6=(\{s_1,    s_3,s_4    \},\{s_2,s_5,s_6\})$,
		$C_3^6=(\{s_1,s_2,s_3        \},\{s_4,s_5,s_6\})$,
		$C_4^6=(\{s_1,    s_3,    s_5\},\{s_2,s_4,s_6\})$ and
		$C_5^6=(\{s_1,s_2,   s_4     \},\{s_3,s_5,s_6\})$.
		As can be seen in
		the below table every arc in $A(D')$ belongs to at least two of the five dicuts, so ${\rm mac}(D') \geq 2w(D')/5$.
		
		\begin{center}
			\begin{tabular}{|c|c|c||c|c|c|c|c|c|c|c|c|c|c|c|c|c|c|} \hline
				\multicolumn{3}{|c||}{Dicut $C_i^6=(X,Y)$} & \multicolumn{15}{|c|}{Contains the following arcs} \\
				\multicolumn{3}{|c||}{ }            & \tX{$s_1$} & \tX{$s_2$} & \tX{$s_3$} & \tX{$s_4$} & \tX{$s_5$} & \tX{$s_1$} & \sX{$s_2$} & \sX{$s_3$} & \sX{$s_4$} & \sX{$s_1$} & \tX{$s_2$} & \sX{$s_3$} & \sX{$s_1$} & \sX{$s_2$} & \sX{$s_1$}   \\ 
				$i$ & $X$           & $Y$           & \tX{$s_2$} & \tX{$s_3$} & \tX{$s_4$} & \tX{$s_5$} & \tX{$s_6$} & \tX{$s_3$} & \sX{$s_4$} & \sX{$s_5$} & \sX{$s_6$} & \sX{$s_4$} & \tX{$s_5$} & \sX{$s_6$} & \sX{$s_5$} & \sX{$s_6$} & \sX{$s_6$}   \\ \hline \hline
				1   & $s_1,s_2,s_5$ & $s_3,s_4,s_6$ &            &   \tX{+}   &            &            &   \tX{+}   &   \tX{+}   &   \tX{+}   &            &            &   \tX{+}   &            &            &            &  \tX{+}    & \tX{+} \\ \hline
				2   & $s_1,s_3,s_4$ & $s_2,s_5,s_6$ &   \tX{+}   &            &            &   \tX{+}   &            &            &            &  \tX{+}    &  \tX{+}    &            &            &  \tX{+}    &  \tX{+}    &            & \tX{+} \\ \hline
				3   & $s_1,s_2,s_3$ & $s_4,s_5,s_6$ &            &            &   \tX{+}   &            &            &            &   \tX{+}   &  \tX{+}    &            &   \tX{+}   &  \tX{+}    &  \tX{+}    &  \tX{+}    &  \tX{+}    & \tX{+} \\ \hline
				4   & $s_1,s_3,s_5$ & $s_2,s_4,s_6$ &   \tX{+}   &            &   \tX{+}   &            &   \tX{+}   &            &            &            &            &   \tX{+}   &            &  \tX{+}    &            &            & \tX{+} \\ \hline
				5   & $s_1,s_2,s_4$ & $s_3,s_5,s_6$ &            &   \tX{+}   &            &   \tX{+}   &            &   \tX{+}   &            &            &  \tX{+}    &            &  \tX{+}    &            &  \tX{+}    &  \tX{+}    & \tX{+} \\ \hline
			\end{tabular}
		\end{center}
		
		We now consider the case when $\nu=7$. Define the following dicuts.
		
		\begin{center}
			\begin{tabular}{|c||c|c|c|c|c|c|c|} \hline
				Dicut $C_i^7=(X,Y)$ & \multicolumn{7}{|c|}{$X$ contains vertices} \\Â 
				& $s_1$ & $s_2$ & $s_3$ & $s_4$ & $s_5$ & $s_6$ & $s_7$ \\ \hline \hline
				$C_1^7=(\{s_1,s_2,s_3,    s_5     \},\{s_4,s_6,s_7\})$       &   +   &   +   &   +   &       &   +   &       &       \\ \hline 
				$C_2^7=(\{s_1,s_2,s_3,    s_6 \},\{s_4,s_5,s_7\})$   &   +   &   +   &   +   &       &       &   +   &       \\ \hline
				$C_3^7=(\{s_1,s_2,    s_4,s_5     \},\{s_3,s_6,s_7\})$       &   +   &   +   &       &   +   &   +   &       &       \\ \hline
				$C_4^7=(\{s_1,s_2,    s_4         \},\{s_3,s_5,s_6,s_7\})$   &   +   &   +   &       &   +   &       &       &       \\ \hline
				$C_5^7=(\{s_1,s_2,            s_6 \},\{s_3,s_4,s_5,s_7\})$   &   +   &   +   &       &       &       &   +   &       \\ \hline
				$C_6^7=(\{s_1,    s_3,s_4,    s_6 \},\{s_2,s_5,s_7\})$       &   +   &       &   +   &   +   &       &   +   &       \\ \hline
				$C_7^7=(\{s_1,    s_3,s_4         \},\{s_2,s_5,s_6,s_7\})$   &   +   &       &   +   &   +   &       &       &       \\ \hline
				$C_8^7=(\{s_1,    s_3,    s_5     \},\{s_2,s_4,s_6,s_7\})$       &   +   &       &   +   &       &   +   &       &       \\ \hline
			\end{tabular}
		\end{center}
		
		Note that $\sum_{i=1}^8 w(C_i^7) \geq 3w(D')$, so ${\rm mac}(D') \geq 3w(D')/8$.
		
		We finally consider the case when $\nu=8$. Define the following dicuts.
		
		\begin{center}
			\begin{tabular}{|c||c|c|c|c|c|c|c|c|} \hline
				Dicut $C_i^8=(X,Y)$ & \multicolumn{8}{|c|}{$X$ contains vertices} \\Â 
				& $s_1$ & $s_2$ & $s_3$ & $s_4$ & $s_5$ & $s_6$ & $s_7$ & $s_8$ \\ \hline \hline
				$C_1^8=   (\{s_1,s_2,s_3,    s_5         \},\{s_4,s_6,s_7,s_8\})$   &   +   &   +   &   +   &       &   +   &       &       & \\ \hline
				$C_2^8=   (\{s_1,s_2,s_3,        s_6     \},\{s_4,s_5,s_7,s_8\})$   &   +   &   +   &   +   &       &       &   +   &       & \\ \hline
				$C_3^8=   (\{s_1,s_2,s_3,            s_7 \},\{s_4,s_5,s_6,s_8\})$   &   +   &   +   &   +   &       &       &       &   +   & \\ \hline
				$C_4^8=   (\{s_1,s_2,    s_4,s_5         \},\{s_3,s_6,s_7,s_8\})$   &   +   &   +   &       &   +   &   +   &       &       & \\ \hline
				$C_5^8=   (\{s_1,s_2,    s_4,    s_6     \},\{s_3,s_5,s_7,s_8\})$   &   +   &   +   &       &   +   &       &   +   &       & \\ \hline
				$C_6^8=   (\{s_1,s_2,    s_4,        s_7 \},\{s_3,s_5,s_6,s_8\})$   &   +   &   +   &       &   +   &       &       &   +   & \\ \hline
				$C_7^8=   (\{s_1,s_2,         s_5,s_6     \},\{s_3,s_4,s_7,s_8\})$   &   +   &   +   &       &       &   +   &   +   &       & \\ \hline
				$C_8^8=   (\{s_1,    s_3,s_4,s_5         \},\{s_2,s_6,s_7,s_8\})$   &   +   &       &   +   &   +   &   +   &       &       & \\ \hline
				$C_9^8=   (\{s_1,    s_3,s_4,    s_6     \},\{s_2,s_5,s_7,s_8\})$   &   +   &       &   +   &   +   &       &   +   &       & \\ \hline
				$C_{10}^8=(\{s_1,    s_3,s_4,        s_7 \},\{s_2,s_5,s_6,s_8\})$   &   +   &       &   +   &   +   &       &       &   +   & \\ \hline
				$C_{11}^8=(\{s_1,    s_3,    s_5,    s_7 \},\{s_2,s_4,s_6,s_8\})$   &   +   &       &   +   &       &   +   &       &   +   & \\ \hline
			\end{tabular}
		\end{center}
		
		Note that $\sum_{i=1}^{11} w(C_i^8) \geq 4w(D')$, so ${\rm mac}(D') \geq 4w(D')/11$.

		\begin{figure}[th]
			\begin{center}
				\begin{tabular}{|c|c|} \hline
					\tikzstyle{vertexA}=[circle,draw, minimum size=8pt, scale=0.9, inner sep=0.9pt]
					\begin{tikzpicture}[scale=0.6]
						\node (s1) at (1,2) [vertexA]{$s_1$};
						\node (s2) at (3,2) [vertexA]{$s_2$};
						\node (s3) at (5,2) [vertexA]{$s_3$};
						\draw[->, line width=0.03cm] (s1) to (s2); 
						\draw[->, line width=0.03cm] (s2) to (s3);
						\node at (0,0) {\mbox{Â }};
						\node at (6,4) {\mbox{Â }};
					\end{tikzpicture} & 
					\tikzstyle{vertexA}=[circle,draw, minimum size=8pt, scale=0.9, inner sep=0.9pt]
					\begin{tikzpicture}[scale=0.6]
						\node (s1) at (3,3) [vertexA]{$s_1$};
						\node (s2) at (1,2) [vertexA]{$s_2$};
						\node (s3) at (3,0.5) [vertexA]{$s_3$};
						\node (s4) at (5.5,0.5) [vertexA]{$s_4$};
						\node (s5) at (7.5,2) [vertexA]{$s_5$};
						\node (s6) at (5.5,3) [vertexA]{$s_6$};
						
						\draw[->, line width=0.07cm] (s1) to (s2);
						\draw[->, line width=0.03cm] (s2) to (s3);
						\draw[->, line width=0.03cm] (s3) to (s4);
						\draw[->, line width=0.03cm] (s4) to (s5);
						\draw[->, line width=0.03cm] (s5) to (s6);
						
						\draw[->, line width=0.03cm] (s1) to (s3);
						\draw[->, line width=0.03cm] (s2) to (s4);
						\draw[->, line width=0.03cm] (s2) to (s5);
						\draw[->, line width=0.03cm] (s3) to (s5);
						\node at (0,0) {\mbox{Â }};
						\node at (8.5,4) {\mbox{Â }};
					\end{tikzpicture} \\ \hline
					\tikzstyle{vertexA}=[circle,draw, minimum size=8pt, scale=0.9, inner sep=0.9pt]
					\begin{tikzpicture}[scale=0.6]
						\node (s1) at (1,3) [vertexA]{$s_1$};
						\node (s2) at (2,1) [vertexA]{$s_2$};
						\node (s3) at (4,1) [vertexA]{$s_3$};
						\node (s4) at (5,3) [vertexA]{$s_4$};
						\draw[->, line width=0.03cm] (s1) to (s2);
						\draw[->, line width=0.03cm] (s2) to (s3);
						\draw[->, line width=0.03cm] (s3) to (s4);
						\draw[->, line width=0.03cm] (s1) to (s3);
						\node at (0,0) {\mbox{Â }};
						\node at (6,4) {\mbox{Â }};
					\end{tikzpicture} &  
					\tikzstyle{vertexA}=[circle,draw, minimum size=8pt, scale=0.9, inner sep=0.9pt]
					\begin{tikzpicture}[scale=0.6]
						\node (s1) at (3,3) [vertexA]{$s_1$};
						\node (s2) at (1,2.5) [vertexA]{$s_2$};
						\node (s3) at (1,0.5) [vertexA]{$s_3$};
						\node (s4) at (4.25,0.5) [vertexA]{$s_4$};
						\node (s5) at (7.5,0.5) [vertexA]{$s_5$};
						\node (s6) at (7.5,2.5) [vertexA]{$s_6$};
						\node (s7) at (5.5,3) [vertexA]{$s_7$};

						\draw[->, line width=0.03cm] (s1) to (s2);
						\draw[->, line width=0.03cm] (s2) to (s3);
						\draw[->, line width=0.03cm] (s3) to (s4);
						\draw[->, line width=0.03cm] (s4) to (s5);
						\draw[->, line width=0.03cm] (s5) to (s6);
						\draw[->, line width=0.03cm] (s6) to (s7);
						
						\draw[->, line width=0.03cm] (s2) to (s4);
						\draw[->, line width=0.03cm] (s4) to (s6);
						\node at (0,0) {\mbox{Â }};
						\node at (8.5,4) {\mbox{Â }};
					\end{tikzpicture} \\ \hline
					\tikzstyle{vertexA}=[circle,draw, minimum size=8pt, scale=0.9, inner sep=0.9pt]
					\begin{tikzpicture}[scale=0.6]
						\node (s1) at (1,3) [vertexA]{$s_1$};
						\node (s2) at (2,1) [vertexA]{$s_2$};
						\node (s3) at (3,2.5) [vertexA]{$s_3$};
						\node (s4) at (4,1) [vertexA]{$s_4$};
						\node (s5) at (5,3) [vertexA]{$s_5$};
						
						\draw[->, line width=0.03cm] (s1) to (s2);
						\draw[->, line width=0.03cm] (s2) to (s3);
						\draw[->, line width=0.03cm] (s3) to (s4);
						\draw[->, line width=0.03cm] (s4) to (s5);
						\draw[->, line width=0.03cm] (s2) to (s4);
						\node at (0,0) {\mbox{Â }};
						\node at (6,4) {\mbox{Â }};
					\end{tikzpicture}  & 
					\tikzstyle{vertexA}=[circle,draw, minimum size=8pt, scale=0.9, inner sep=0.9pt]
					\begin{tikzpicture}[scale=0.6]
						\node (s1) at (1,3) [vertexA]{$s_1$};
						\node (s2) at (3,3) [vertexA]{$s_2$};
						\node (s3) at (1,1.75) [vertexA]{$s_3$};
						\node (s4) at (2.6,0.5) [vertexA]{$s_4$};
						\node (s5) at (5.9,0.5) [vertexA]{$s_5$};
						\node (s6) at (7.5,1.75) [vertexA]{$s_6$};
						\node (s7) at (5.5,3) [vertexA]{$s_7$};
						\node (s8) at (7.5,3) [vertexA]{$s_8$};

						\draw[->, line width=0.07cm] (s1) to (s2);
						\draw[->, line width=0.07cm] (s2) to (s3);
						\draw[->, line width=0.07cm] (s3) to (s4);
						\draw[->, line width=0.07cm] (s4) to (s5);
						\draw[->, line width=0.07cm] (s5) to (s6);
						\draw[->, line width=0.07cm] (s6) to (s7);
						\draw[->, line width=0.07cm] (s7) to (s8);
						
						\draw[->, line width=0.07cm] (s2) to (s4);
						\draw[->, line width=0.07cm] (s5) to (s7);
						\draw[->, line width=0.03cm] (s2) to (s5);
						\draw[->, line width=0.03cm] (s3) to (s5);
						\draw[->, line width=0.03cm] (s4) to (s6);
						\draw[->, line width=0.03cm] (s4) to (s7);
						\node at (0,0) {\mbox{Â }};
						\node at (8.5,4) {\mbox{Â }};
					\end{tikzpicture}  \\ \hline
				\end{tabular}
				\caption{Digraphs giving upper bounds on $c_\nu$ for $\nu \in \{3,4,5,6,7,8\}$. The thick arcs have weight two and all other arcs have weight one.} \label{fig:upperBcr}
		\end{center} \end{figure}
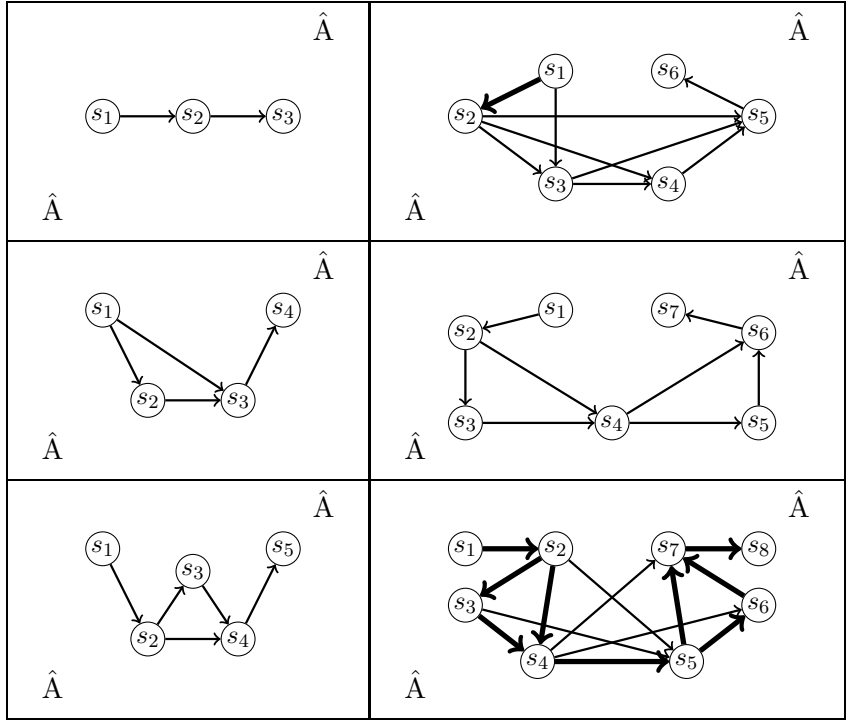

	\end{proof}

	Recall that $c_2 = 1$, $c_3=c_4=0.5$, $c_5=c_6 = 0.4$, $c_7=0.375$ and $c_8=\frac{4}{11} \approx 0.363636$.
	Using a computer, we can also show that $c_9 = \frac{13}{37} \approx 0.35135$ and $c_{10} = \frac{9}{26} \approx 0.34615$ and $c_{11} = \frac{31}{92} \approx 0.33696$.
	
\end{document}